\documentclass[12pt,a4paper,oneside]{amsart}

\usepackage{amsthm}
\usepackage{amsmath}
\usepackage{amssymb}
\usepackage{amscd}
\usepackage[utf8]{inputenc}
\usepackage{typearea}
\usepackage{eufrak}
\usepackage{yfonts}
\usepackage{mathrsfs}
\usepackage{hyperref}
\usepackage{xcolor}
\usepackage{mathtools}

\newcommand{\alp}{\boldsymbol{a}}
\newcommand{\RR}{\mathbb{R}}
\newcommand{\CC}{\mathbb{C}}
\newcommand{\ZZ}{\mathbb{Z}}
\newcommand{\NN}{\mathbb{N}}

\newcommand{\TSS}[1][A]{\X_{#1}}
\newcommand{\osh}[1][X]{\sigma_{#1}}
\newcommand{\tsh}[1][\X]{\mathbf{\sigma}_{#1}}

\newcommand{\lp}{\operatorname{lp}}
\newcommand{\G}{\mathcal{G}}
\newcommand{\LL}{\mathcal{L}}
\renewcommand{\hom}{\operatorname{Hom}}

\newcommand{\x}{\mathbf{x}}
\newcommand{\X}{\mathbf{X}}
\newcommand{\y}{\mathbf{y}}
\newcommand{\Y}{\mathbf{Y}}
\newcommand{\bvarphi}{\boldsymbol{\varphi}}

\newcommand{\D}{\mathcal{D}}
\newcommand{\Oo}{\mathcal{O}}
\newcommand{\frt}{\operatorname{frt}}

\newtheorem{lemma}{Lemma}[section]
\newtheorem{corollary}[lemma]{Corollary}
\newtheorem{theorem}[lemma]{Theorem}
\newtheorem{proposition}[lemma]{Proposition}
\newtheorem{claim}[lemma]{Claim}
\theoremstyle{definition}

\newtheorem{remark}[lemma]{Remark}

\title[Flow equivalence and orbit equivalence for shifts of finite type]{Flow equivalence and orbit equivalence for shifts of finite type and isomorphism of their groupoids}
 
\author{Toke Meier Carlsen}
\address{Department of Science and Technology\\University of the Faroe Islands\\
N\'oat\'un 3\\ FO-100 T\'orshavn\\the Faroe Islands}
\email{toke.carlsen@gmail.com}

\author{Søren Eilers}
\address{Department of Mathematical Sciences\\University of Copenhagen\\Universitetsparken 5\\DK-2100 Copenhagen\\Denmark}
\email{eilers@math.ku.dk}

\author{Eduard Ortega}
\address{Department of Mathematical Sciences\\NTNU\\NO-7491 Trondheim\\Norway}
\email{eduardo.ortega@math.ntnu.no}

\author{Gunnar Restorff}
\address{Department of Science and Technology\\University of the Faroe Islands\\
N\'oat\'un 3\\ FO-100 T\'orshavn\\the Faroe Islands}
\email{gunnarr@setur.fo} 
 
\date{\today}

\keywords{Continuous orbit equivalence, flow equivalence, shift spaces, subshift, shifts of finite type, groupoids, cohomology of topological dynamical systems, cohomology of étale groupoids, graph $C^*$-algebras, Leavitt path algebras, Cuntz--Krieger algebras, diagonal-preserving isomorphisms.}

\subjclass[2010]{Primary 37B10; Secondary 16S99, 22A22, 37A55, 46L55}

\begin{document}

\begin{abstract}
We give conditions for when continuous orbit equivalence of one-sided shift spaces implies flow equivalence of the associated two-sided shift spaces. Using groupoid techniques, we prove that this is always the case for shifts of finite type. This generalises a result of Matsumoto and Matui from the irreducible to the general case. We also prove that a pair of one-sided shift spaces of finite type are continuously orbit equivalent if and only if their groupoids are isomorphic, and that the corresponding two-sided shifts are flow equivalent if and only if the groupoids are stably isomorphic.
	
As applications we show  that two finite directed graphs with no sinks and no sources are move equivalent if and only if the corresponding graph $C^*$-algebras are stably isomorphic by a diagonal-preserving isomorphism (if and only if the corresponding Leavitt path algebras are stably isomorphic by a diagonal-preserving isomorphism), and that two topological Markov chains are flow equivalent if and only if there is a diagonal-preserving isomorphism between the stabilisations of the corresponding Cuntz--Krieger algebras (the latter generalises a result of Matsumoto and Matui about irreducible topological Markov chains with no isolated points to a result about general topological Markov chains).

We also show that for general shift spaces, strongly continuous orbit equivalence implies two-sided conjugacy.
\end{abstract}

\maketitle

\section{Introduction}

In their beautiful recent paper \cite{MM}, Matsumoto and Matui proved that a simple Cuntz--Krieger algebra remembers the  flow equivalence class of the irreducible shift of finite type defining it, provided that the canonical diagonal subalgebra is considered as a part of the data. A key tool for obtaining this groundbreaking result was the realisation that diagonal-preserving isomorphism translates directly to isomorphism of the groupoids associated to the shift spaces, reducing the problem to establishing that when two one-sided irreducible shifts of finite type are continuously orbit equivalent in the sense developed by Matsumoto, then the corresponding two-sided shift spaces are flow equivalent.

Having such rigidity results for $C^*$-algebras associated to general shift spaces of finite type would provide a better understanding of the classification problem for general Cuntz--Krieger algebras recently solved in \cite{ERRS} and \cite{ERRS2}. From the point of view of symbolic dynamics, it is also of interest to determine the class of shift spaces for which continuous orbit equivalence implies flow equivalence. 

The groupoid component of the proof in \cite{MM} has in \cite{BCW} and \cite{CRS} been generalised to a much more general setting, but the argument leading from diagonal-preserving isomorphism to flow equivalence in \cite{MM} goes through a deep result about the ordered cohomology of irreducible shifts of finite type by Boyle and Handelman (\cite{BH}) which does not readily extend to the reducible case. In addition, several of the arguments used in \cite{MM} rely on the assumption that the shifts of finite type in question do not contain isolated periodic points.

In the present paper we give a direct proof that continuously orbit equivalent shifts of finite type are also flow equivalent (Theorem~\ref{ib}) and thereby generalising \cite[Theorem 3.5]{MM} from irreducible one-sided Markov shifts to general (possible reducible) shifts of finite type. We do that by producing a concrete flow equivalence from a given orbit equivalence between general shift spaces with continuous cocycles under added hypotheses on the given orbit equivalence and cocycles (Proposition~\ref{diego}), and then proving by methods related to the original proof in \cite{MM} that when the shift spaces are of finite type, then these hypotheses may always be arranged (Proposition~\ref{sven} and Proposition~\ref{john}). 

As a corollary to Proposition~\ref{diego}, we generalise in Corollary~\ref{cor:scoe} \cite[Theorem 5.5]{Mat2} from irreducible topological Markov chains with no isolated points to general shift spaces by showing that for general shift spaces, strongly continuous orbit equivalence implies two-sided conjugacy.

We also prove that the groupoids of two one-sided shifts of finite type are isomorphic if and only if the shift spaces are continuously orbit equivalent (Theorem~\ref{orbit}), and by combining this with a result of Matui \cite{Matui} and results in \cite{CRS} and \cite{ERRS}, we obtain that these groupoids are stably isomorphic if and only if the corresponding two-sided shift spaces are flow equivalent (Theorem~\ref{thm:1}).

As applications, we show in Corollary~\ref{cor:2} that the one-sided edge shifts of two finite directed graphs with no sinks and no sources are continuous orbit equivalent if and only if the corresponding graph $C^*$-algebras are isomorphic by a diagonal-preserving isomorphism (if and only if the corresponding Leavitt path algebras are isomorphic by a diagonal-preserving  isomorphism), and we show in Corollary~\ref{cor:1} that the graphs are move equivalent, as defined in \cite{Sor}, if and only if the corresponding graph $C^*$-algebras are stably isomorphic by a diagonal-preserving isomorphism (if and only if the corresponding Leavitt path algebras are stably isomorphic by a diagonal-preserving  isomorphism). 

We also apply our results to Cuntz--Krieger algebras and topological Markov chains and directed graphs of $\{0,1\}$-matrices and thereby generalise \cite[Theorem 2.3]{MM} and \cite[Corollary 3.8]{MM} from the irreducible to the general case (Corollary~\ref{CK} and Corollary~\ref{cor:3}).

\subsection*{Acknowledgements}

This work was partially supported by the Danish National Research Foundation
through the Centre for Symmetry and Deformation (DNRF92), by VILLUM
FONDEN through the network for Experimental Mathematics in Number Theory,
Operator Algebras, and Topology, and by the Danish Council for Independent Research $|$
Natural Sciences (7014-00145B).

Some of the work was done while all four authors
were attending the research program \emph{Classification of operator algebras: complexity,
rigidity, and dynamics} at the Mittag-Leffler Institute, January--April 2016. We
thank the institute and its staff for the excellent work conditions provided.

\section{Definitions and notation}

In this section we briefly recall the definitions of \emph{shift spaces}, \emph{shifts of finite type}, \emph{continuous orbit equivalence of shift spaces}, and \emph{flow equivalence of shift spaces}, and introduce notation.

We let $\NN$ denote the set of positive integers, and $\NN_0$ the set of non-negative integers. 

\subsection{One-sided shift spaces}

A \emph{one-sided shift space} (or \emph{one-sided subshift}) is a closed, and hence compact, subset $X$ of $\alp^{\NN_0}$, where $\alp$ is a finite set equipped with the discrete topology and  $\alp^{\NN_0}$ is equipped with the product topology, such that $X$ is invariant by the shift transformation 
$$\sigma:\alp^{\NN_0}\to \alp^{\NN_0}\,,$$
(i.e., $\sigma(X)= X$) given by $(\sigma((x_i)_{i\in\NN_0}))_j=x_{j+1}$ for $j\in\NN_0$. When $X$ is a one-sided shift space, then we let $\osh:X\to X$ denote the restriction of $\sigma$ to $X$. For $n\in\NN_0$ we denote by $\osh^n$ the $n$-fold composition of $\osh$ with itself (when $n=0$, then $\osh^n$ denotes the identity map on $X$). 

Two one-sided shift spaces $X$ and $Y$ are \emph{conjugate} if there is a \emph{conjugacy} between them, i.e., a homeomorphism $h:X\to Y$ such that $\osh[Y]\circ h=h\circ\osh$.

Let $X$ be a one-sided shift space. We say that $x\in X$ is \emph{periodic} if $\osh^p(x)=x$ for some $p\in\NN$, and that $x$ is \emph{eventually periodic} if $\osh^n(x)$ is periodic for some $n\in\NN_0$. When $x\in X$ is eventually periodic, then we call the number 
$$\lp(x):=\min\{p\in\NN:\exists n,m\in\NN_0:p=n-m\text{ and }\osh^n(x)=\osh^m(x)\}$$ 
\emph{the least period} of $x$. 

When $X$ is a shift space, we write $\LL(X)$ for the \emph{language of $X$} (i.e., the set of finite words, included the empty word $\emptyset$, that appear in elements of $X$). Given a word $v$ in $\LL(X)$, we denote by $|v|$ the length of $v$, and for $m\in\NN$, we let $\LL^m(X)$ be the set of words in $\LL(X)$ of length $m$. Given $x\in X$ and $n,m\in \NN_0$ with $n\leq m$, we define the word $x_{[n,m]}:=(x_n,\ldots,x_m)\in \LL^{m-n+1}(X)$.
 For $v\in\LL(X)\setminus\{\emptyset\}$, we write $Z(v)$ for the \emph{cylinder set} $\{x\in X:x_{[0,|v|)}=v\}$ where $x_{[0,|v|)}:=x_{[0,|v|-1]}$.

\subsection{Shifts of finite type}

A \emph{one-sided shift of finite type} is a one-sided shift space $X$ such that there is an $m\in\NN$ with the property that if $v\in\LL(X)$ has length $m$ and $uv,vw\in\LL(X)$, then $uvw\in\LL(X)$. The shift map $\osh$ is a local homeomorphism if and only if $X$ is a shift of finite type, in which case $\osh^n$ is a local homeomorphism for all $n\in\NN_0$.

\subsection{Continuous orbit equivalence}

Let $X$ and $Y$ be two one-sided shift spaces. Following \cite{Mat}, we say that a homeomorphism $h:X\to Y$ is a \emph{continuous orbit equivalence} if there exist continuous maps $k,l:X\to \NN_0$ and $k',l':Y\to \NN_0$  such that 
\begin{equation}\label{eq:1}
\osh[Y]^{k(x)}(h(\osh(x)))=\osh[Y]^{l(x)}(h(x))
\end{equation}
for $x\in X$, and
\begin{equation}\label{eq:2}
\osh^{k'(y)}(h^{-1}(\osh[Y](y)))=\osh^{l'(y)}(h^{-1}(y))
\end{equation}

for $y\in Y$. Observe that in this case $h^{-1}:Y\rightarrow X$ is also a continuous orbit equivalence. We say that $X$ and $Y$ are \emph{continuously orbit equivalent} if there exists a continuous orbit equivalence $h:X\to Y$ (it is routine to check that the composition of two continuous orbit equivalences is a continuous orbit equivalence, and thus that continuous orbit equivalence indeed is an equivalence relation of one-sided shift spaces, cf.~\cite[Lemma 2.3]{Mat2}). If $h:X\to Y$ is a continuous orbit equivalence, then we say that a pair $(k,l)$ of continuous maps $k,l:X\to \NN_0$ satisfying \eqref{eq:1} is a \emph{$h$-cocycle pair}. 

\subsection{Flow equivalence}

Let $X$ be a one-sided shift space. Given $\x=(x_n)_{n\in \ZZ}\in \alp^\ZZ$ and $m\in \ZZ$, we define 
$$\x_{[m,\infty)}:=(x_m,x_{m+1},\ldots)\in\alp^{\NN_0}\,.$$
The \emph{two-sided shift space associated to $X$} is defined to be
$$\X:=\{\x\in\alp^\ZZ: \x_{[m,\infty)}\in X\text{ for all }m\in \ZZ\}\,.$$ 
The set $\X$ is a closed and compact subset of $\alp^\ZZ$ with the induced product topology of $\alp^\ZZ$, and invariant by the shift transformation $$\tsh:\X\to \X$$ given by $(\tsh((x_i)_{i\in\ZZ})_j=x_{j+1}$ for $j\in\ZZ$. Notice that $X\leftrightarrow \X$ is a bijective correspondence between the class of one-sided shift spaces and the class of two-sided shift spaces (i.e., the class of closed subsets $\X$ of $\alp^\ZZ$ satisfying that $\tsh(\X)=\X$). Two two-sided shift spaces $\X$ and $\Y$ are \emph{conjugate} if there is a \emph{conjugacy} between them, i.e., a homeomorphism $\bvarphi:\X\to \Y$ such that $\tsh[\Y]\circ \bvarphi=\bvarphi\circ\tsh$. If $X$ and $Y$ are conjugate, then $\X$ and $\Y$ are conjugate (but $\X$ and $\Y$ can be conjugate without $X$ and $Y$ being conjugate).

We say that $\x\in\X$ is \emph{periodic} if $\tsh^p(\x)=\x$ for some $p\in\NN$. When $\x\in \X$ is periodic, then we call the number $$\lp(\x):=\min\{p\in\NN:\tsh^p(\x)=\x\}$$ \emph{the least period} of $\x$.

Let $\sim$ be the smallest equivalence relation on $\X\times\RR$ such that $(\tsh^n(\x),t)\sim (\x,t+n)$ for $\x\in\X$, $t\in\RR$ and $n\in\ZZ$, and let $[(\x,t)]$ denote the equivalence class of $(\x,t)$.
The \emph{suspension} $S\X$ of $\X$ is the quotient $\X\times\RR/\sim$ equipped with the quotient topology of the product topology on $\X\times\RR$.

A \emph{flow equivalence} between the suspensions of two two-sided shift spaces $\X$ and $\Y$ is a homeomorphism $\psi:S\X\to S\Y$ that maps flow lines onto flow lines in an orientation preserving way: so if $\x\in\X$, $\y\in\Y$, $r,s,t,u\in\RR$, $s,u>0$ and $\psi([(\x,t)])=[(\y,r)]$, then there is an $v>0$ such that $\psi([(\x,t+s)])=[(\y,r+v)]$, and a $w>0$ such that $\psi^{-1}([(\y,r+u)])=[(\x,t+w)]$. Two two-sided shift spaces $\X$ and $\Y$ are \emph{flow equivalent} if there exists a flow equivalence
between $S\X$ and $S\Y$. It is routine to check that the composition of two flow equivalences is a flow equivalence, and thus that flow equivalence is an equivalence relation of two-sided shift spaces. If $\X$ and $\Y$ are conjugate, then $\X$ and $\Y$ are flow equivalent, but $\X$ and $\Y$ can be flow equivalent without being conjugate.

\subsection{The cohomology of a shift space} \label{cohomology}

Let $X$ be a one-sided shift space. Following \cite{MM}, we let $H^X$ be the group 
\begin{equation*}
H^X:=C(X,\ZZ)/\{f-f\circ\osh:f\in C(X,\ZZ)\}
\end{equation*}
with addition defined by $[f]+[g]=[f+g]$, and we let 
\begin{equation*}
H^X_+:=\{[f]\in H^X: f(x)\ge 0\text{ for all }x\in X\}.	
\end{equation*}
It follows from \cite[Lemma 3.1]{MM} that the preordered group $(H^X,H^X_+)$ is isomorphic to the ordered cohomology group $(G^{\tsh},G^{\tsh}_+)$ of $(\X,\tsh)$ defined in \cite{BH} (\cite[Lemma 3.1]{MM} is only stated for irreducible shifts associated with $\{0,1\}$ matrices, but it is easy to see that its proof holds for general shift spaces).

\subsection{The groupoid of a one-sided shift space of finite type}

The groupoid $\G_X$ of a one-sided shift of finite type $X$ has unit space $\G^{(0)}:=X$ and morphisms  
\begin{equation*}
\G_X:=\{(x,n,x')\in X\times\ZZ\times X:\exists i,j\in\NN_0:n=i-j \text{ and } \osh^i(x)=\osh^j(x')\}.
\end{equation*}
The range and source maps $r,s:\G_X\to \G^{(0)}$ are defined by $r((x,n,x'))=x$ and $s((x,n,x'))=x'$, and the product and inverse operators by $(x,n,x')(x',n',x'')=(x,n+n',x'')$ and $(x,n,x')^{-1}=(x',-n,x)$. We let $c:\G_X\to\ZZ$ be the map defined by $c((x,n,x'))=n$. There is a topology on $\G_X$ that has a basis consisting of sets of the form
\begin{equation*}
\{(x,i-j,x'):x\in U,\ x'\in U',\ \osh^i(x)=\osh^j(x')\}
\end{equation*}
where $i,j\in\NN_0$ and $U$ and $U'$ are open subsets such that $\osh^i$ restricted to $U$ is injective, $\osh^j$ restricted to $U'$ is injective, and $\osh^i(U)=\osh^j(U')$. If we identify $X$ with the subspace $\{(x,0,x):x\in X\}$ of $\G_X$, then the topology of $X$ coincides with the subspace topology.   

With the topology described above, $\G_X$ is an \emph{ample} Hausdorff groupoid, i.e., the product and inverse operators are continuous and the topology is Hausdorff and has a basis of compact open bisections (a subset $A$ of a groupoid $\G$ is a \emph{bisection} if both the restriction of the range map and the restriction of the source map to $A$ are injective). In particular, $\G_X$ is étale (i.e., the range and source maps are local homeomorphisms).
 
As in \cite{MM}, we let $\hom(\G_X,\ZZ)$ be the set of continuous maps $\omega:\G_X\to\ZZ$ such that $\omega(\eta^{-1})=-\omega(\eta)$ for $\eta\in\G_X$ and $\omega(\eta_1\eta_2)=\omega(\eta_1)+\omega(\eta_2)$ for $\eta_1,\eta_2\in\G_X$ with $s(\eta_1)=r(\eta_2)$. For $f\in C(X,\ZZ)$, the map $\partial(f):\G_X\to\ZZ$ defined by $\partial(f)(\eta)=f(r(\eta))-f(s(\eta))$ belongs to $\hom(\G_X,\ZZ)$. As in \cite{MM}, we denote by $H^1(\G_X)$ the group
\begin{equation*}
H^1(\G_X):=\hom(\G_X,\ZZ)/\{\partial(f):f\in C(X,\ZZ)\}
\end{equation*}
with addition defined by $[f]+[g]=[f+g]$. We shall in Proposition \ref{erik} see that there is an isomorphism $\Phi:H^1(\G_X)\to H^X$ such that $\Phi([f])=[g]$, where $g\in C(X,\ZZ)$ is given by $g(x)=f((x,1,\osh(x)))$, and $\Phi([f])\in H^X_+$ if and only if $f((x,\lp(x),x))\ge 0$ for every eventually periodic point $x\in X$, cf.~\cite[Proposition 3.4]{MM}. 

A \emph{homomorphism} between two topological groupoids $\G_1$ and $\G_2$ is a continuous map $\phi:\G_1\to\G_2$ such that $\phi(\eta^{-1})=\phi(\eta)^{-1}$ for every $\eta\in\G_1$, and $\phi(\eta_1)\phi(\eta_2)$ is defined and equal to $\phi(\eta_1\eta_2)$ for all $\eta_1,\eta_2\in\G_1$ for which $\eta_1\eta_2$ is defined. An \emph{isomorphism} between two topological groupoids $\G_1$ and $\G_2$ is a bijective homomorphism $\phi:\G_1\to\G_2$ such that $\phi^{-1}:\G_2\to\G_1$ is also a homomorphism.

\section{Orbit equivalence and flow equivalence for general shift spaces}

One of the goals of this paper is to show that continuous orbit equivalence implies flow equivalence for shifts of finite type, and thereby generalise \cite[Theorem 3.5]{MM} from irreducible one-sided Markov shifts to general (possible reducible) shifts of finite type. In this section we prove Proposition~\ref{diego}, which gives sufficient conditions for when continuous orbit equivalence implies flow equivalence for general shift spaces. These conditions are related to the preordered cohomology groups of one-sided shift spaces introduced in Section~\ref{cohomology} (see the discussion right after Remark~\ref{Olga}). As a corollary (Corollary~\ref{cor:scoe}), we generalise \cite[Theorem 5.5]{Mat2} and show that for general shift spaces, strongly continuous orbit equivalence implies two-sided conjugacy.

In this paper, we only apply Proposition~\ref{diego} to shifts of finite type, but we hope that it also can be used to prove that orbit equivalence implies flow equivalence for other classes of shift spaces. Our strategy for proving Proposition~\ref{diego} is to use techniques and ideas related to those used in \cite{MM} and \cite{MM2} to construct a discrete flow equivalence from a continuous orbit equivalence satisfying the conditions of Proposition~\ref{diego}, and then construct a flow equivalence from the discrete flow equivalence. Since we work with shift spaces that might not be irreducible and might contain isolated points, we have to modify the approach of \cite{MM} and \cite{MM2} a bit.

\subsection{A sufficient condition for flow equivalence}

Let $X$ and $Y$ be two one-sided shift spaces and let $h:X\to Y$ be a continuous orbit equivalence. We say that $h$ \emph{maps eventually periodic points to eventually periodic points} if $h(x)$ is eventually periodic exactly when $x$ is eventually periodic.

\begin{remark}
Matsumoto and Matui prove in \cite[Proposition 3.5]{MM2} that if $X$ and $Y$ are the one-sided shift spaces associated with two irreducible $\{0,1\}$ square matrices that satisfy the Condition (I) introduced by Cuntz and Krieger in \cite{CK}, then any continuous orbit equivalence between $X$ and $Y$ maps eventually periodic points to eventually periodic points. By inspecting the proof, one sees that it actually holds for any pair of one-sided shifts spaces $X$ and $Y$ that have the property that the complement of the set of eventually periodic points is dense. We prove in Proposition~\ref{sven} that any continuous orbit equivalence between shifts of finite type maps eventually periodic points to eventually periodic points. We do not know if there are continuous orbit equivalences between one-sided shift spaces that do not map eventually periodic points to eventually periodic points.
\end{remark}
 
Let $X$ and $Y$ be two one-sided shift spaces and let $h:X\to Y$ be a continuous orbit equivalence that maps eventually periodic points to eventually periodic points. We say that an $h$-cocycle pair $(k,l)$ is \emph{least period preserving} if 
$$\lp(h(x))=\sum_{i=0}^{\lp(x)-1}\bigl(l(\osh^i(x))-k(\osh^i(x))\bigr)$$ 
for every eventually periodic point $x\in X$ (this terminology is justified by Proposition~\ref{john}).

\begin{proposition}\label{diego}
Let $X$ and $Y$ be two one-sided shift spaces and suppose that $h:X\to Y$ is a continuous orbit equivalence that maps eventually periodic points to eventually periodic points, that $(k,l)$ is a least period preserving $h$-cocycle pair, that $(k',l')$ is a least period preserving $h^{-1}$-cocycle pair, and that $b:X\to\ZZ$, $n:X\to\NN_0$, $b':Y\to\ZZ$ and $n':Y\to\NN_0$ are continuous maps such that $l(x)-k(x)=n(x)+b(x)-b(\osh(x))$ and $l'(y)-k'(y)=n'(y)+b'(y)-b'(\osh[Y](y))$ for $x\in X$ and $y\in Y$. Then $\X$ and $\Y$ are flow equivalent. 
\end{proposition}

The rest of this section contains the proof of Proposition~\ref{diego}. We assume in the rest of this section that $X$, $Y$, $h$, $k$, $l$, $k'$, $l'$, $b$, $b'$, $n$, and $n'$ are as specified in the proposition. We shall construct an explicit flow equivalence $\psi:S\X\to S\Y$ from this data. 

We begin by constructing a continuous map $\varphi:X\rightarrow Y$ and establish some properties of it in Claim~\ref{lemma1} and Claim~\ref{yap}. In Claim~\ref{lemma_n} we show that the map $n$ satisfies a condition which we need in order to construct a continuous map $\bvarphi:\X\to \Y$ in Claim~\ref{proposition1}. We then prove some properties of $\bvarphi$ in Claim~\ref{lemma2}, Claim~\ref{robert} and Claim~\ref{periodic}, before we for each $\x\in\X$ construct an increasing piecewise linear homeomorphism $r_\x:\RR\to\RR$. In Claim~\ref{harrison} we show a relationship between $r_\x$ and $r_{\tsh^p(\x)}$, before we in Claim~\ref{super} finally show that there is a flow equivalence $\psi:S\X\to S\Y$ given by $\psi([(\x,t)])=[(\bvarphi(\x),r_\x(t))]$.

Since $l$ and $b$ are bounded, we can by adding a constant to $b$ if necessary, assume that $b(x)\ge l(x)$ for every $x\in X$. Similarly, we can assume that $b'(y)\ge l'(y)$ for every $y\in Y$. 

We let $\varphi:X\rightarrow Y$ be the continuous map defined by 
\begin{equation}\label{function_1}
\varphi(x)=\osh[Y]^{b(x)}(h(x))
\end{equation}
for $x\in X$.

\begin{claim}\label{lemma1} The function $\varphi$ defined in (\ref{function_1}) is finite-to-one, i.e., $|\varphi^{-1}(y)|<\infty$ for every $y\in Y$.
\end{claim} 

\begin{proof} 
Recall that $h$ is a homeomorphism and $b$ is bounded. Let $j\in\NN_0$ be such that $0\leq b(x)\leq j$ for every $x\in X$. For $y\in Y$ we have that 
$$\varphi^{-1}(y)\subseteq \bigcup_{i=0}^j h^{-1}(\osh[Y]^{-i}(y))\,.$$
Since $\osh[Y]^{i}$ is finite-to-one, so is $\osh[Y]^i\circ h$. It follows that $\varphi$ is finite-to-one.
\end{proof}

\begin{claim}\label{yap}
For $x\in X$ we have that
\begin{equation}\label{equation_2}\varphi(\osh(x))=\osh[Y]^{n(x)}(\varphi(x))\,.
\end{equation}
\end{claim}

\begin{proof}
Since $b(\osh(x))=n(x)-l(x)+k(x)+b(x)$, $n(x)-l(x)+b(x)\ge b(x)-l(x)\geq 0$, and $\osh[Y]^{k(x)}(h(\osh(x)))=\osh[Y]^{l(x)}(h(x))$, it follows that
\begin{align*} \varphi(\osh(x))& =\osh[Y]^{b(\osh(x))}(h(\osh(x)))=\osh[Y]^{n(x)-l(x)+k(x)+b(x)}(h(\osh(x)))\\
& =\osh[Y]^{n(x)-l(x)+b(x)}(\osh[Y]^{k(x)}(h(\osh(x))))=\osh[Y]^{n(x)-l(x)+b(x)}(\osh[Y]^{l(x)}(h(x))) \\
 &  = \osh[Y]^{n(x)+b(x)}(h(x))=\osh[Y]^{n(x)}(\varphi(x))\,.\qedhere
\end{align*}	
\end{proof}

For $j\in\NN$ and $x\in X$, we set $n^j(x):=\sum_{i=1}^{j}n(\osh^{i-1}(x))$ and $n^0(x):=0$.
Observe that then
\begin{equation}\label{eq:56}
	\varphi(\osh^j(x))=\varphi(\osh(\osh^{j-1}(x)))=\osh[Y]^{n(\osh^{j-1}(x))}(\varphi(\osh^{j-1}(x)))=\cdots=\osh[Y]^{n^j(x)}(\varphi(x))\,,
\end{equation}
by an iteration of (\ref{equation_2}).

\begin{claim}\label{lemma_n}
Given $\x\in \X$ and $i_0\in\ZZ$, there exist $i,j\in \ZZ$ such that $i<i_0$ and $n(\x_{[i,\infty)})\neq 0$ and $j>i_0$ and $n(\x_{[j,\infty)})\neq 0$.
\end{claim}

\begin{proof}
We first show that $n(\x_{[j,\infty)})\neq 0$ for some $j>i_0$. Assume, for contradiction, that $n(\x_{[j,\infty)})=0$ for every $j>i_0$. Then $n^{j-i_0}(\x_{[i_0,\infty)})=\sum_{i=i_0}^{j-1}n(\x_{[i,\infty)})=0$ for every $j>i_0$. An application of \eqref{eq:56} gives us that
$$\varphi(\x_{[j,\infty)})=\varphi(\osh^{j-i_0}(\x_{[i_0,\infty)]}))=\osh[Y]^{n^{j-i_0}(\x_{[i_0,\infty)})}(\varphi(\x_{[i_0,\infty)}))=\varphi(\x_{[i_0,\infty)})$$
for every $j>i_0$, and since $\varphi$ is finite-to-one (Claim~\ref{lemma1}), it follows that $\{\x_{[j,\infty)}:j>i_0\}$ is finite, and thus that $\x_{[j,\infty)}$ is periodic for some $j>i_0$. But then
$$\lp(h(\x_{[j,\infty)}))=\sum_{i=0}^{\lp(\x_{[j,\infty)})-1}\left(l(\x_{[i+j,\infty)})-k(\x_{[i+j,\infty)})\right)=\sum_{i=0}^{\lp(\x_{[j,\infty)})-1}n(\x_{[i+j,\infty)})=0\ ,$$
which cannot be the case.

Similarly, if $n(\x_{[i,\infty)})= 0$ for every $i<i_0$, then 
$$\varphi(\x_{[i_0,\infty)})=\varphi(\osh^{i_0-i}(\x_{[i,\infty)]}))=\osh[Y]^{n^{i_0-i}(\x_{[i,\infty)})}(\varphi(\x_{[i,\infty)}))=\varphi(\x_{[i,\infty)})$$
for every $i<i_0$, and since $\varphi$ is finite-to-one, it follows that $\{\x_{[i,\infty)}:i<i_0\}$ is finite, and thus that $\x$ is periodic. It follows from the first part of the proof that there is an $i\in \NN_0$ such that $n(\x_{[i,\infty)})\neq 0$, but since $\x$ is periodic, there is a $j\in\NN$ such that $\x_{[-j,\infty)}=\x_{[i,\infty)}$ from which it follows that $n(\x_{[-j,\infty)})=n(\x_{[i,\infty)})\neq 0$.
\end{proof}

For $\x\in\X$ and $j\in\ZZ$, we set
\begin{equation*}
m_\x(j):=
\begin{cases}
-\sum_{i=1}^{-j} n(\x_{[-i,\infty)})&\text{if }j<0,\\
0&\text{if }j=0,\\
\sum_{i=0}^{j-1} n(\x_{[i,\infty)})&\text{if }j>0.
\end{cases}
\end{equation*}
Then $m_\x:\ZZ\to\ZZ$ is a weakly increasing function (i.e., $m_\x(i)\le m_\x(j)$ if $i<j$), and it follows from Claim~\ref{lemma_n} that $m_\x(j)\to\pm\infty$ for $j\to\pm\infty$.

It is straightforward to check that if $\x\in\X$ and $i,j\in\ZZ$, then
\begin{equation}\label{bobs}
m_\x(i+j)=m_\x(i)+m_{\tsh^i(\x)}(j).
\end{equation}

\begin{claim}\label{proposition1} 
There is a continuous map $\bvarphi:\X\to \Y$ such that $\bvarphi(\x)_{[m_\x(-i),\infty)}=\varphi(\x_{[-i,\infty)})$ for $i\in\NN_0$.
\end{claim}

\begin{proof}
Let $\x\in \X$. Since $m_\x(-i)\to -\infty$ for $i\to \infty$, it follows that there is at most one $\y\in\Y$ such that $\y_{[m_\x(-i),\infty)}=\varphi(\x_{[-i,\infty)})$ for $i\in\NN_0$. That there is such a $\y\in\Y$ follows from the fact that
\begin{equation*} 
	\osh[Y]^{n(\x_{[-i-1,\infty)})}(\varphi(\x_{[-i-1,\infty)}))  =\varphi(\osh(\x_{[-i-1,\infty)})) =\varphi(\x_{[-i,\infty)}) \,
\end{equation*}
for $i\in\NN_0$. 

Since, for fixed $i\in\NN_0$, the function $\x\mapsto m_\x(-i)$ is a continuous and thus locally constant function from $\X$ to $\ZZ$, and $\varphi$ is continuous, it follows that $\bvarphi$ is continuous. 
\end{proof}

\begin{claim}\label{lemma2} 
$\tsh[\Y]^{m_\x(j)}(\bvarphi(\x))=\bvarphi(\tsh^j(\x))$ for $\x\in\X$ and $j\in\ZZ$.
\end{claim}

\begin{proof}
Let $\x'\in\X$ and $i,j'\in\NN_0$. It follows from \eqref{bobs} that
\begin{align*}
	(\tsh[\Y]^{-m_{\x'}(j')}(\bvarphi(\tsh^{j'}(\x'))))_{[m_{\x'}(-i),\infty)}
	&=(\bvarphi(\tsh^{j'}(\x')))_{[m_{\x'}(-i)-m_{\x'}(j'),\infty)}\\
	&=(\bvarphi(\tsh^{j'}(\x')))_{[m_{\tsh^{j'}(\x')}(-i-j'),\infty)}\\
	&=\varphi((\tsh^{j'}(\x')_{[-i-j',\infty)}))\\
	&=\varphi(\x'_{[-i,\infty)})\\
	&=\bvarphi(\x')_{[m_{\x'}(-i),\infty)}.
\end{align*}
Thus, 
\begin{equation}\label{eq:33}
	\tsh[\Y]^{-m_{\x'}(j')}(\bvarphi(\tsh^{j'}(\x')))=\bvarphi(\x').
\end{equation}
If $j\ge 0$, then an application of \eqref{eq:33} with $\x'=\x$ and $j'=j$ gives us that $\tsh[\Y]^{m_\x(j)}(\bvarphi(\x))=\bvarphi(\tsh^j(\x))$, and if $j<0$, then an application of \eqref{eq:33} with $\x'=\tsh^{j}(\x)$ and $j'=-j$ gives us together with \eqref{bobs} that
\begin{equation*}
\bvarphi(\tsh^{j}(\x))
=\tsh[\Y]^{-m_{\tsh^{j}(\x)}(-j)}(\bvarphi(\tsh^{-j}(\tsh^{j}(\x))))
=\tsh[\Y]^{-m_{\tsh^{j}(\x)}(-j)}(\bvarphi(\x))
=\tsh[\Y]^{m_{\x}(j)}(\bvarphi(\x)).
\end{equation*}
\end{proof}

Similarly to how we constructed $\varphi$, $m_\x$ and $\bvarphi$, we can for each $\y\in\Y$ construct a weakly increasing function $m'_\y:\ZZ\to\ZZ$ and continuous functions $\varphi':Y\to X$ and $\bvarphi':\Y\to\X$ such that
\begin{equation*}
m'_\y(j)=
\begin{cases}
-\sum_{i=1}^{-j} n'(\y_{[-i,\infty)})&\text{if }j<0,\\
0&\text{if }j=0,\\
\sum_{i=0}^{j-1} n'(\y_{[i,\infty)})&\text{if }j>0,
\end{cases}
\end{equation*}
$\varphi'(y)=\osh^{b'(y)}(h^{-1}(y))$, and $\bvarphi'(\y)_{[m'_\y(-i),\infty)}=\varphi'(\y_{[-i,\infty)})$ for $y\in Y$, $\y\in\Y$, $i\in\NN_0$, and $j\in\ZZ$.

For $j\in\NN$ and $y\in Y$, we set $(n')^j(y):=\sum_{i=1}^{j}n'(\osh[Y]^{i-1}(y))$ and $(n')^0(y):=0$.

\begin{claim} \label{robert}
Given $\x\in\X$ and $\y\in\Y$, there exist $d,d'\in\ZZ$ such that $\bvarphi'(\bvarphi(\x))=\tsh^d(\x)$ and $\bvarphi(\bvarphi'(\y))=\tsh[\Y]^{d'}(\y)$.
\end{claim}

\begin{proof}
Let $\x\in\X$. Then we have for $j\in\NN_0$ that 
\begin{align*}
	\bvarphi'(\bvarphi(\x))_{[m'_{\bvarphi(\x)}(m_\x(-j)),\infty)}
	&=\varphi'(\bvarphi(\x)_{[m_\x(-j),\infty)})\\
	&=\varphi'(\varphi(\x_{[-j,\infty)}))\\
	&=\varphi'(\osh[Y]^{b(\x_{[-j,\infty)})}(h(\x_{[-j,\infty)})))\\
	&=\osh^{(n')^{b(\x_{[-j,\infty)})}(h(\x_{[-j,\infty)}))}(\varphi'(h(\x_{[-j,\infty)})))\\
	&=\osh^{(n')^{b(\x_{[-j,\infty)})}(h(\x_{[-j,\infty)}))+b'(h(\x_{[-j,\infty)}))}(\x_{[-j,\infty)})\\
	&=\x_{[-j+(n')^{b(\x_{[-j,\infty)})}(h(\x_{[-j,\infty)}))+b'(h(\x_{[-j,\infty)})),\infty)}.
\end{align*}
Let us first set $d=(n')^{b(\x_{[0,\infty)})}(h(\x_{[0,\infty)}))+b'(h(\x_{[0,\infty)}))$. By letting $j=0$ we see that $\bvarphi'(\bvarphi(\x))_{[0,\infty)}=\x_{[d,\infty)}$. Since $m_\x(-j)\to -\infty$ as $j\to\infty$, it follows that 
$$m'_{\bvarphi(\x)}(m_\x(-j))\to -\infty \text{ as } j\to\infty,$$
and since $b$ and $b'$ are bounded functions, and $(n')^i$ is bounded for each $i\in\NN_0$, we get that 
$$-j+(n')^{b(\x_{[-j,\infty)})}(h(\x_{[-j,\infty)}))+b'(h(\x_{[-j,\infty)}))\to -\infty \text{ as } j\to\infty.$$
It follows that if $\x$ is periodic, then $\bvarphi'(\bvarphi(\x))$ is also periodic, and $\bvarphi'(\bvarphi(\x))=\tsh^d(\x)$.

Suppose then that $\x$ is not periodic. Then there is a $j\in\NN_0$ such that 
$$\bvarphi'(\bvarphi(\x))_{[m'_{\bvarphi(\x)}(m_\x(-j)),\infty)}=\x_{[-j+(n')^{b(\x_{[-j,\infty)})}(h(\x_{[-j,\infty)}))+b'(h(\x_{[-j,\infty)})),\infty)},$$
is not periodic. It follows that if we now set 
$$d=-m'_{\bvarphi(\x)}(m_\x(-j))-j+(n')^{b(\x_{[-j,\infty)})}(h(\x_{[-j,\infty)}))+b'(h(\x_{[-j,\infty)})),$$
then $\bvarphi'(\bvarphi(\x))=\tsh^d(\x)$.

That there for $\y\in\Y$ is a $d'\in\ZZ$ such that $\bvarphi(\bvarphi'(\y))=\tsh[\Y]^{d'}(\y)$, can be proved in a similar way.
\end{proof}

\begin{claim} \label{periodic}
Let $\x\in\X$. Then $\bvarphi(\x)$ is periodic if and only if $\x$ is, in which case $\lp(\bvarphi(\x))=m_{\x}(\lp(\x))$.
\end{claim}

\begin{proof}
Suppose $\x$ is periodic with period $p$. Since $m_\x(j)$ goes monotonically to $\infty$ as $j\to\infty$, it follows from \eqref{bobs} that $m_\x(p)\ne 0$. It thus follows from Claim~\ref{lemma2} that $\bvarphi(\x)$ is periodic with period $m_\x(p)$. Analogously, if $\bvarphi(\x)$ is periodic with period $q$, then $\x$ is periodic with period  $m'_{\bvarphi(\x)}(q)$.

Suppose again that $\x$ is periodic. Then $\x_{[0,\infty)}$ is also periodic.
Since $h$ maps eventually periodic points to eventually periodic points, it follows that $h(\x_{[0,\infty)})$ is eventually periodic. It is clear that $\lp(\x_{[0,\infty)})=\lp(\x)$ and $\lp(h(\x_{[0,\infty)}))=\lp(\bvarphi(\x))$. Since the $h$-cocycle pair $(k,l)$ is least period preserving, it follows that
\begin{align*}
\lp(\bvarphi(\x))&=\lp(h(\x_{[0,\infty)}))
=\sum_{i=0}^{\lp(\x_{[0,\infty)})-1}(l(\osh^i(\x_{[0,\infty)}))-k(\osh^i(\x_{[0,\infty)})))\\
&=\sum_{i=0}^{\lp(\x)-1}n(\x_{[i,\infty)})=m_{\x}(\lp(\x)).\qedhere
\end{align*}
\end{proof}

Let $\x\in\X$. Let functions $i_\x,j_\x:\RR\to\ZZ$ be given by $i_\x(t):=\max\{i\le t:n(\x_{[i,\infty)})\ne 0\}$ and $j_\x(t):=\min\{j> t:n(\x_{[j,\infty)})\ne 0\}$ (it follows from Claim~\ref{lemma_n} that $i_\x(t)$ and $j_\x(t)$ are well-defined), and let
\begin{equation*}
	r_\x(t):=m_\x(i_\x(t))+\frac{t-i_\x(t)}{j_\x(t)-i_\x(t)}n(\x_{[i_\x(t),\infty)}).
\end{equation*}
Then $r_\x:\RR\to\RR$ is an increasing piecewise linear homeomorphism such that $r_\x(i)=m_\x(i)$ for those $i\in\ZZ$ for which $n(\x_{[i,\infty)})\ne 0$.

\begin{claim} \label{harrison}
$r_\x(t+p)=r_{\tsh^p(\x)}(t)+m_\x(p)$ for $\x\in\X$, $t\in\RR$ and $p\in\ZZ$.
\end{claim}

\begin{proof}
Since $i_\x(t+p)=i_{\tsh^p(\x)}(t)+p$ and $j_\x(t+p)=j_{\tsh^p(\x)}(t)+p$, it follows from \eqref{bobs} that
\begin{equation*}
r_\x(t+p)=r_{\tsh^p(\x)}(t)+m_\x(i_{\tsh^p(\x)}(t)+p)-m_{\tsh^p(\x)}(i_{\tsh^p(\x)}(t))
=r_{\tsh^p(\x)}(t)+m_\x(p).\qedhere
\end{equation*} 
\end{proof}

It is now routine to construct a flow equivalence $\psi:S\X\to S\Y$ from $\bvarphi$ and $r_\x$ (cf. \cite{BCE} and \cite{PS}). 

\begin{claim} \label{super}
There is a flow equivalence $\psi:S\X\to S\Y$ such that 
$$\psi([(\x,t)])=[(\bvarphi(\x),r_\x(t))]$$ 
for $\x\in\X$ and $t\in\RR$.
\end{claim} 

\begin{proof}
It follows from Claim~\ref{lemma2} and Claim~\ref{harrison} that
\begin{align*}
	[(\bvarphi(\tsh^p(\x)),r_{\tsh^p(\x)}(t))]
	&=[(\tsh[\Y]^{m_\x(p)}(\bvarphi(\x)),r_{\tsh^p(\x)}(t))]\\
	&=[(\bvarphi(\x),r_{\tsh^p(\x)}(t)+m_\x(p))]\\
	&=[(\bvarphi(\x),r_\x(t+p))].
\end{align*}
It follows that there is a map $\psi:S\X\to S\Y$ such that $\psi([(\x,t)])=[(\bvarphi(\x),r_\x(t))]$ for $\x\in\X$ and $t\in\RR$.

We check that $\psi$ is injective. Suppose $\psi([(\x,t)])=\psi([(\x',t')])$. Then there is a $p\in\ZZ$ such that $\bvarphi(\x)=\tsh[\Y]^p(\bvarphi(\x'))$ and $r_\x(t)+p=r_{\x'}(t')$. It then follows from Claim~\ref{lemma2} and Claim~\ref{robert} that there is a $q\in\ZZ$ such that $\x'=\tsh^q(\x)$. So $[(\x',t')]=[(\x,s)]$ for some $s\in\RR$. If $\x$ is not periodic, then $\bvarphi(\x)$ is not periodic either, so $\psi([(\x,s)])=\psi([(\x,t)])$ implies that $r_\x(s)=r_\x(t)$, and since $r_\x$ is injective, it follows that $s=t$ and thus that $[(\x',t')]=[(\x,s)]=[(\x,t)]$. Suppose that $\x$ is periodic. Then it follows from Claim~\ref{periodic} that $\bvarphi(\x)$ is periodic and that $\lp(\bvarphi(\x))=m_\x(\lp(\x))$. So $\psi([(\x,s)])=\psi([(\x,t)])$ implies that $r_\x(s)=r_\x(t)+i\ m_\x(\lp(\x))$ for some $i\in\ZZ$. It follows from Claim~\ref{harrison} that $r_\x(t)+i\ m_\x(\lp(\x))=r_\x(t+i\lp(\x))$, and since $r_\x$ is injective, it follows that $s=t+i\lp(\x)$ and thus that $[(\x',t')]=[(\x,s)]=[(\x,t+i\lp(\x))]=[(\x,t)]$.

Next, we show that $\psi$ is surjective. Let $[(\y,s)]\in S\Y$. It follows from Claim~\ref{robert} that $[(\y,s)]=[(\bvarphi(\bvarphi'(\y)),r)]$ for some $r\in\RR$. Since $r_{\bvarphi'(\y)}$ is surjective, it follows that there is a $t\in\RR$ such that $\psi([(\bvarphi'(\y),t)])=[(\bvarphi(\bvarphi'(\y)),r)]=[(\y,s)]$.

Let us then show that $\psi$ is continuous. It suffices to show that the map $(\x,t)\mapsto (\bvarphi(\x),r_\x(t))$ is a continuous map from $\X\times\RR$ to $\Y\times\RR$. Let $(\x_i,t_i)$ be a sequence that converges to $(\x,t)$ in $\X\times\RR$. Then $\x_i\to\x$ in $\X$. Since $\bvarphi$ is continuous, it follows that $\bvarphi(\x_i)\to\bvarphi(\x)$. Since the map $n$ is continuous, it follows that there is an $M\in\NN$ such that $i_{\x_i}(s)=i_\x(s)$ and $j_{\x_i}(s)=j_\x(s)$ for $i\ge M$ and $s\in (t-1,t+1)$, and thus that there is an $N\in\NN$ such that $r_{\x_i}(s)=r_\x(s)$ for $i\ge N$ and $s\in (t-1,t+1)$. Since $r_\x$ is continuous, it follows that $r_{\x_i}(t_i)\to r_\x(t)$. Thus, $(\bvarphi(\x_i),r_{\x_i}(t_i))\to (\bvarphi(\x),r_\x(t))$.

We have now shown that $\psi$ is bijective and continuous. Since $S\X$ is compact and $S\Y$ is Hausdorff, it follows that $\psi$ is a homeomorphism. Since $r_\x$ is an increasing homeomorphism from $\RR$ to $\RR$, it follows that $\psi$ maps flow lines onto flow lines in an orientation preserving way. So $\psi$ is a flow equivalence.
\end{proof}

\subsection{Strongly continuous orbit equivalence}
Following \cite{Mat2}, we say that two one-sided shift spaces $X$ and $Y$ are \emph{strongly continuous orbit equivalent} if there is a continuous orbit equivalence $h:X\to Y$, an $h$-cocycle pair $(k,l)$, and a continuous map $b:X\to\ZZ$ such that 
\begin{equation*}
l(x)-k(x)=1+b(x)-b(\osh(x))
\end{equation*}
for all $x\in X$.

Matsumoto proved in \cite[Theorem 5.5]{Mat2} that if two irreducible topological Markov chains $X$ and $Y$ with no isolated points are strongly continuous orbit equivalent, then the corresponding two-sided shift spaces $\X$ and $\Y$ are conjugate. We now generalise this results to arbitrary shift spaces.

\begin{corollary}\label{cor:scoe}
If two shift spaces $X$ and $Y$ are strongly continuous orbit equivalent, then the corresponding two-sided shift spaces $\X$ and $\Y$ are conjugate.
\end{corollary}

\begin{proof}
If $X$ and $Y$ are strongly continuous orbit equivalent, then we can choose the function $n:X\to\NN$ in Proposition~\ref{diego} to be constantly equal to $1$. Then $m_\x(j)=j$, $i_\x(j)=j$, $j_\x(j)=j+1$, and $r_\x(j)=j$ for all $\x\in\X$ and all $j\in\ZZ$. Consequently, $\bvarphi:\X\to\Y$ is a conjugacy.
\end{proof}

\section{Orbit equivalence and flow equivalence for shifts of finite type}

In this section we use Proposition~\ref{diego} to prove the following theorem.

\begin{theorem}\label{ib}
Suppose $X$ and $Y$ are one-sided shifts of finite type and that they are continuously orbit equivalent. Then $\X$ and $\Y$ are flow equivalent.
\end{theorem}

If $X$ and $Y$ are irreducible, then the result of Theorem~\ref{ib} easily follows from \cite[Theorem 3.5]{MM} and the fact that every one-sided shift of finite type is conjugate to a one-sided topological Markov shift.

\begin{remark} \label{Olga}
It follows from \cite[Theorem 1.5]{BH} and \cite[Lemma 3.1]{MM} that if $\X$ and $\Y$ are flow equivalent, then there is an isomorphism from $H^X$ to $H^Y$ that maps $H^X_+$ onto $H^Y_+$. Theorem~\ref{ib} can therefore be seen as a generalisation of \cite[Theorem 3.5]{MM} (it will also follow directly from Proposition~\ref{sven} and Proposition~\ref{erik} that if $X$ and $Y$ are continuously orbit equivalent, then there is an isomorphism from $H^X$ to $H^Y$ that maps $H^X_+$ onto $H^Y_+$).
\end{remark}

To prove Theorem~\ref{ib} we will prove that if $X$ and $Y$ are one-sided shifts of finite type and $h:X\to Y$ is a continuous orbit equivalence, then there exist functions $k$, $l$, $k'$, $l'$, $b$, $b'$, $n$, and $n'$ with the property specified in Proposition~\ref{diego}. We do this by closely following \cite{MM} and use the groupoid of a one-sided shift of finite type. However, since we are working with general shifts of finite type and not just irreducible shifts of finite type with no isolated periodic points as in \cite{MM}, we cannot just simply follow the approach of \cite{MM}. In particular, the possibility that our shift spaces contain isolated periodic points implies that we need to make adjustments to the approach used in \cite{MM} (see Proposition~\ref{sven} and Remark~\ref{vm}).

The conditions in Proposition~\ref{diego} are equivalent to the condition that there is an isomorphism $\phi:\G_X\to\G_Y$ such that $r(\phi(\eta))=h(r(\eta))$ and $s(\phi(\eta))=h(s(\eta))$ for $\eta\in\G_X$ and  $\phi((x,\lp(x),x))=(h(x),\lp(h(x)),h(x))$ for every eventually periodic point $x\in X$, and such that $\phi$ induces an isomorphism from $H^Y$ to $H^X$ that maps the class of the constant function 1 into $H^X_+$. We show in Proposition~\ref{sven} that $h$ maps  eventually periodic points to eventually periodic points and that there is an isomorphism $\phi:\G_X\to\G_Y$ such that $r(\phi(\eta))=h(r(\eta))$ and $s(\phi(\eta))=h(s(\eta))$ for $\eta\in\G_X$ and  $\phi((x,\lp(x),x))=(h(x),\lp(h(x)),h(x))$ for every eventually periodic point $x\in X$, and then we generalise \cite[Proposition 3.4]{MM} in Proposition~\ref{erik} and show that there is an isomorphism from $H^1(\G_X)$ to $H^X$ that maps the class of a function $f\in \hom(\G_X,\ZZ)$ into $H^X_+$ if and only if $f((x,\lp(x),x))\ge 0$ for every eventually periodic point $x\in X$. From this we deduce in Proposition~\ref{john} that if $\phi:\G_X\to\G_Y$ is an isomorphism with the above mentioned properties, then there exist functions $k$, $l$, $k'$, $l'$, $b$, $b'$, $n$, and $n'$ with the property specified in Proposition~\ref{diego}. We end the section by putting it all together and give the proof of Theorem~\ref{ib}. 

We begin with two lemmas which we need for the proof of Proposition~\ref{sven}.

\begin{lemma}\label{kurt}
Let $X$ be a one-sided shift of finite type.	
\begin{enumerate}
\item If $x$ is an isolated point in $X$, then $x$ is eventually periodic.
\item If $x$ is not an isolated point in $X$, $U$ is an open neighbourhood of $x$, $W$ is an open subset of $X$, $\alpha:U\to W$ is a homeomorphism, $k,l:U\to\NN_0$ are continuous, and $\osh^{k(x')}(\alpha(x'))=\osh^{l(x')}(x')$ for every $x'\in U$, then there is a unique $n\in\ZZ$ with the property that there exist $k_0,l_0\in\NN_0$ and an open subset $V$ such that $n=l_0-k_0$, $x\in V\subseteq U$ and $\osh^{k_0}(\alpha(x'))=\osh^{l_0}(x')$ for every $x'\in V$.
\end{enumerate}
\end{lemma}

\begin{proof}
(1): Suppose $x$ is an isolated point in $X$. Because $X$ is a shift of finite type, there is an $m\in\NN$ such that if $v\in\LL(X)$ has length $m$ and $uv,vw\in\LL(X)$, then $uvw\in\LL(X)$. Choose $n$ such that $Z(x_{[0,n-1]})=\{x\}$. Since there are only finitely many words of length $m$ in $\LL(X)$, it follows that there are $p,q\in\NN$ such that $p\ge n$, $q-p\ge m$ and $x_{[p,p+m-1]}=x_{[q,q+m-1]}$. Since $q-p\ge m$, the infinite sequence $x_{[0,p-1]}x_{[p,q-1]}x_{[p,q-1]}x_{[p,q-1]}\dots$ belongs to $X$ and thus to $Z(x_{[0,n-1]})$, so it must be equal to $x$. This shows that $x$ is eventually periodic.

(2): Let $x,U,W,k,l,\alpha$ be given as specified.  
We first show the existence of an $n\in\ZZ$, $k_0,l_0\in\NN_0$ and an open subset $V$ such that $n=l_0-k_0$, $x\in V\subseteq U$ and $\osh^{k_0}(\alpha(x'))=\osh^{l_0}(x')$ for every $x'\in V$. Let $k_0:=k(x)$, $l_0:=l(x)$ and $n:=l_0-k_0$. Since $k,l:U\to\NN_0$ are continuous, there is an open subset $V$ such that $x\in V\subseteq U$ and $\osh^{k_0}(\alpha(x'))=\osh^{l_0}(x')$ for every $x'\in V$. 

Suppose then that $n'\in\ZZ$, $k_0',l_0'\in\NN_0$ and $V'$ is an open subset such that $n\ne n'=l_0'-k_0'$, $x\in V'\subseteq U$ and $\osh^{k_0'}(\alpha(x'))=\osh^{l_0'}(x')$ for every $x'\in V'$. Let $U':=V\cap V'$, $k_0'':=\max\{k_0,k_0'\}$, $h:=l_0+k_0''-k_0$ and $j:=l_0'+k_0''-k_0'$. Then $U'$ is open, $x\in U'\subseteq U$, $h\ne j$ and $\osh^h(x')=\osh^{k_0''}(\alpha(x'))=\osh^j(x')$ for every $x'\in U'$. Let  $p=\max\{h,j\}$ and $q=\min\{h,j\}$. Then $p>q$ because $h\ne j$. Choose $r\ge p$ such that $Z(x_{[0,r-1]})\subseteq U'$. Then $x'=x_{[0,p-1]}x_{[q,p-1]}x_{[q,p-1]}\dots $ for every $x'\in Z(x_{[0,r-1]})$, but this contradicts the assumption that $x$ is not an isolated point in $X$. 
\end{proof}

\begin{lemma}\label{henrik}
Let $X$ and $Y$ be two one-sided shifts of finite type. Suppose $\phi:\G_X\to\G_Y$ is an isomorphism and $h:X\to Y$ is a homeomorphism such that $\phi((x',0,x'))=(h(x'),0,h(x'))$ for all $x'\in X$. 

If $x\in X$ is eventually periodic, then $h(x)$ is eventually periodic and $\phi((x,\lp(x),x))$ is either equal to $(h(x),\lp(h(x)),h(x))$ or to $(h(x),-\lp(h(x)),h(x))$. If $x$ is not isolated in $X$, then $\phi((x,\lp(x),x))=(h(x),\lp(h(x)),h(x))$.
\end{lemma}

\begin{proof}
The proof uses ideas from \cite[Lemma 3.3]{MM}. Suppose $x\in X$ is eventually periodic. Since $\phi$ is an isomorphism and $\phi((x',0,x'))=(h(x'),0,h(x'))$ for all $x'\in X$, it follows that $\phi((x,\lp(x),x))=(h(x),n,h(x))$ for some $n\in\ZZ$ different from 0. It follows that $h(x)$ is eventually periodic.
	
Since $\phi$ is an isomorphism, it follows that either $n=\lp(h(x))$ or $n=-\lp(h(x))$. Suppose $n=-\lp(h(x))$. We will show that $x$ is then isolated in $X$.
	
Choose $m\in\NN$ such that if $v\in\LL(X)$ has length $m$ and $uv,vw\in\LL(X)$, then $uvw\in\LL(X)$, and choose $r,s\in\NN_0$ such that $r-s=\lp(x)$ and $\osh^r(x)=\osh^s(x)$. Then 
\begin{equation*}
	A:=\{(x'',\lp(x),x'):x''\in Z(x_{[0,r+m-1]}),\ x'\in Z(x_{[0,s+m-1]}),\ \osh^r(x'')=\osh^s(x')\}
\end{equation*}
is an open bisection containing $(x,\lp(x),x)$. It follows that $s(A)=Z(x_{[0,s+m-1]})$ and $r(A)=Z(x_{[0,r+m-1]})$ and the map $\alpha_A:s(A)\to r(A)$ defined by $\alpha_A(s(\xi))=r(\xi)$ for $\xi\in A$ is a homeomorphism (cf.~\cite[Proposition 3.3]{BCW}) such that
\begin{equation}\label{eq:3}
\alpha_A(x')=x_{[0,r+m-1]}\osh^{s+m}(x')	
\end{equation}
for $x'\in s(A)$. Notice that $r(A)\subseteq s(A)$. It follows from \eqref{eq:3} that $\lim_{i\to\infty}\alpha_A^i(x')= x$ for all $x'\in s(A)$.

Choose $m'\in\NN$ such that if $v\in\LL(Y)$ has length $m'$ and $uv,vw\in\LL(Y)$, then $uvw\in\LL(y)$. Since $\phi((x,\lp(x),x))=(h(x),-\lp(h(x)),h(x))$, there is an $j\in\NN$ such that $\osh[Y]^j(h(x))=\osh[Y]^{j+\lp(h(x))}(h(x))$, and such that the open bisection
\begin{multline*}
\{(y'',-\lp(h(x)),y'):y''\in Z(h(x)_{[0,j+m'-1]}),\\ y'\in Z(h(x)_{[0,j+\lp(h(x))+m'-1]}),\ \osh[Y]^j(y'')=\osh[Y]^{j+\lp(h(x))}(y')\}
\end{multline*} 
is contained in $\phi(A)$. 

Let $y\in h(s(A))$. Then $\lim_{i\to\infty}\alpha_A^i(h^{-1}(y))= x$. It follows that there is an $I\in\NN$ such that $h(\alpha_A^i(h^{-1}(y)))\in Z(h(x)_{[0,j+\lp(h(x))+m'-1]})$ for $i\ge I$. Let $y':=h(\alpha_A^I(h^{-1}(y)))$ and $y'':=h(x)_{[0,j-1]}\osh[Y]^{j+\lp(h(x))}(y')$. Then $(y'',-\lp(h(x)),y')\in\phi(A)$. It follows that $y''=h(\alpha_A(h^{-1}(y')))\in Z(h(x)_{[0,j+\lp(h(x))+m'-1]})$, and thus that $y'\in Z(h(x)_{[0,j+2\lp(h(x))+m'-1]})$. By repeating this argument, we see that $y'\in Z(h(x)_{[0,j+i\lp(h(x))+m'-1]})$ for all $i\in\NN$. It follows that $y'=h(x)$ and thus that $y=h(x)$. This shows that $h(x)$ is isolated in $Y$. Since $h$ is a homeomorphism, it follows that $x$ is isolated in $X$.
\end{proof}

\begin{proposition}\label{sven}
Let $X$ and $Y$ be two one-sided shifts of finite type and let $h:X\to Y$ be a continuous orbit equivalence. Then $h$ maps eventually periodic points to eventually periodic points, and there is an isomorphism $\phi:\G_X\to \G_Y$ such that $r(\phi(\eta))=h(r(\eta))$ and $s(\phi(\eta))=h(s(\eta))$ for $\eta\in\G_X$ and such that $\phi((x,\lp(x),x))=(h(x),\lp(h(x)),h(x))$ for every eventually periodic point $x\in X$.
\end{proposition}

\begin{proof}
We begin by constructing the isomorphism $\phi:\G_X\to \G_Y$. We first define what $\phi(\eta)$ is when $s(\eta)$ is an isolated point in $X$, and then what $\phi(\eta)$ is when $s(\eta)$ is not an isolated point in $X$.
	 
For $x\in X$, let $[x]:=\{x'\in X:\exists\eta\in\G_X\text{ such that }r(\eta)=x\text{ and }s(\eta)=x'\}$. Notice that if $x'\in [x]$, then $x$ is isolated in $X$ if and only if $x'$ is. It follows from Lemma~\ref{kurt}(1) that if $x$ is an isolated point in $X$, then $[x]$ contains a periodic point. Choose for each $A\in\{[x]:x\text{ is an isolated point in } X\}$, a periodic point $x_A\in A$. For $x\in A$, let $j_x:=\min\{j\in\NN_0:\osh^j(x)=x_A\}$. If the source of $\eta\in\G_X$ is an isolated point, then $r(\eta)$ is also an isolated point, $[r(\eta)]=[s(\eta)]$, and $j_{r(\eta)}-j_{s(\eta)}-c(\eta)=n\lp(x_{[r(\eta)]})$ for some $n\in\ZZ$. We write $n_\eta$ for this $n$.

We similarly let $[y]:=\{y'\in y:\exists\eta\in\G_Y\text{ such that }r(\eta)=y\text{ and }s(\eta)=y'\}$ for $y\in Y$, choose for each $A\in\{[y]:y\text{ is an isolated point in } Y\}$ a periodic point $y_A\in A$, let $j_y:=\min\{j\in\NN_0:\osh[Y]^j(y)=y_A\}$ for $y\in A$, and let $n_\eta$ be the unique integer such that $j_{r(\eta)}-j_{s(\eta)}-c(\eta)=n_\eta\lp(y_{[r(\eta)]})$ for $\eta\in\G_Y$ with $s(\eta)$ an isolated point in $Y$. 

Let $\eta\in\G_X$ and suppose $s(\eta)$ is an isolated point in $X$. Then $r(\eta)$ is also an isolated point in $X$, $h(r(\eta))$ and $h(s(\eta))$ are isolated points in $Y$, and 
$$(h(r(\eta)),j_{h(r(\eta))}-j_{h(s(\eta))}-n_\eta\lp(y_{[h(r(\eta))]}),h(s(\eta)))\in\G_Y.$$ 
We let   
$$\phi(\eta):=(h(r(\eta)),j_{h(r(\eta))}-j_{h(s(\eta))}-n_\eta\lp(y_{[h(r(\eta))]}),h(s(\eta))).$$

Suppose then that $s(\eta)$ is not an isolated point in $X$. Choose an open bisection $A$ such that $\eta\in A$. Then $s(A)$ and $r(A)$ are open in $X$ and the map $\alpha_A:s(A)\to r(A)$ defined by $\alpha_A(s(\xi))=r(\xi)$ for $\xi\in A$ is a homeomorphism with the property that there are continuous maps $k,l:s(A)\to\NN_0$ such that $\osh^{k(x)}(\alpha_A(x))=\osh^{l(x)}(x)$ for every $x\in s(A)$ (cf.~\cite[Proposition 3.3]{BCW}). Since $h:X\to Y$ is a continuous orbit equivalence, it follows that there are a homeomorphism $\alpha'_A:h(s(A))\to h(r(A))$ such that $\alpha'_A(h(x))=h(\alpha(x))$ for $x\in s(A)$, and continuous maps $k',l':h(s(A))\to\NN_0$ such that $\osh[Y]^{k'(y)}(\alpha'_A(y))=\osh[Y]^{l'(y)}(y)$ for every $y\in h(s(A))$ (cf.~the proof of \cite[Proposition 3.4]{BCW}). Since $h(s(\eta))$ is not an isolated point in $Y$, it follows from Lemma~\ref{kurt}(2) that there is a unique $n\in\ZZ$ such that $\osh[Y]^{k_0}(\alpha'_A(y))=\osh[Y]^{l_0}(y)$ for all $y$ in some open neighbourhood $V\subseteq h(s(A))$ of $h(s(\eta))$ and some $k_0,l_0\in\NN_0$ satisfying $l_0-k_0=n$. Then $(h(r(\eta)),n,h(s(\eta)))\in\G_Y$. Notice that $n$ does not depend on the particular choice of $A$ because if $B$ is another open bisection containing $\eta$, then $A\cap B$ is also an open bisection containing $\eta$ and $\alpha_A(x)=\alpha_{A\cap B}(x)=\alpha_B(x)$ for every $x\in s(A\cap B)$, so if $n'$ and $n''$ are integers such that $\osh[Y]^{k'_0}(\alpha'_B(y))=\osh[Y]^{l'_0}(y)$ for all $y$ in some open neighbourhood $V'\subseteq h(s(B))$ of $h(s(\eta))$ and some $k'_0,l'_0\in\NN_0$ satisfying $l'_0-k'_0=n'$, and $\osh[Y]^{k''_0}(\alpha'_B(y))=\osh[Y]^{l''_0}(y)$ for all $y$ in some open neighbourhood $V''\subseteq h(s(A\cap B))$ of $h(s(\eta))$ and some $k''_0,l''_0\in\NN_0$ satisfying $l''_0-k''_0=n''$, then it follows from the uniqueness of $n$, $n'$ and $n''$ that $n=n''=n'$. We let $\phi(\eta):=(h(r(\eta)),n,h(s(\eta)))$.

We have now defined a map $\phi:\G_X\to \G_Y$ such that $r(\phi(\eta))=h(r(\eta))$ and $s(\phi(\eta))=h(s(\eta))$ for $\eta\in\G_X$ and such that $\phi((x,\lp(x),x))=(h(x),\lp(h(x)),h(x))$ for every isolated point in $X$. It is straightforward to check that $\phi$ is a bijection, that $\phi(\eta)^{-1}=\phi(\eta^{-1})$ for every $\eta\in\G_X$, and that $\phi(\eta_1\eta_2)=\phi(\eta_1)\phi(\eta_2)$ for $\eta_1,\eta_2\in\G_X$ with $s(\eta_1)=r(\eta_2)$. 

Since $x\in X$ is eventually periodic if and only if $\{\eta\in\G_X: s(\eta)=r(\eta)=x\}$ is infinite, and $h(x)$ is eventually periodic if and only if $\{\eta\in\G_Y: s(\eta)=r(\eta)=h(x)\}$ is infinite, it follows that $h$ maps eventually periodic points to eventually periodic points.

We will now show that $\phi$ is continuous. We will do that by, for $\eta\in\G_X$ and an open subset neighbourhood $V$ of $\phi(\eta)$, constructing an open neighbourhood $U$ of $\eta$ such that $\phi(U)\subseteq V$. If $s(\eta)$ is an isolated point in $X$, then $\eta$ is an isolated point in $\G_X$, so we can just take $U$ to be equal to $\{\eta\}$ in that case. Suppose that $s(\eta)$ is not an isolated point in $X$. Choose $m'\in\NN$ such that if $v\in\LL(Y)$ has length $m'$ and $uv,vw\in\LL(Y)$, then $uvw\in\LL(Y)$, and choose $k',l'\in\NN$ such that $\osh[Y]^{k'}(r(\phi(\eta)))=\osh[Y]^{l'}(s(\phi(\eta)))$ and
\begin{equation*}
\phi(\eta)\in\{(r(\phi(\eta))_{[0,k'-1]}y,k'-l',s(\phi(\eta))_{[0,l'-1]}y):y\in Z(r(\phi(\eta))_{[k',k'+m'-1]})\}\subseteq V.
\end{equation*}
Then choose $m\in\NN$ such that if $v\in\LL(X)$ has length $m$ and $uv,vw\in\LL(X)$, then $uvw\in\LL(X)$, and choose $k,l\in\NN$ such that $\osh^k(r(\eta))=\osh^l(s(\eta))$, $k-l=c(\eta)$, $h(Z(r(\eta)_{[0,k-1]}))\subseteq Z(r(\phi(\eta))_{[0,k'+m'-1]})$, and $h(Z(s(\eta)_{[0,l-1]}))\subseteq Z(s(\phi(\eta))_{[0,l'+m'-1]})$. Then 
\begin{equation*}
A:=\{(r(\eta)_{[0,k-1]}x,k-l,s(\eta)_{[0,l-1]}x):x\in Z(r(\eta)_{[k,k+m-1]})\}
\end{equation*}
is a bisection that contains $\eta$. Choose an open neighbourhood $V'\subseteq h(s(A))$ of $h(s(\eta))$ and $n'\in\NN_0$ such that $\osh[Y]^{k'+n'}(h(\alpha_A(h^{-1}(y)))=\osh[Y]^{l'+n'}(y)$ for all $y\in V'$, and $n\in\NN_0$ such that $h(Z(s(\eta)_{[0,l+n-1]}))\subseteq V'$, and let 
\begin{equation*}
U:=\{(r(\eta)_{[0,k+n-1]}x,k-l,s(\eta)_{[0,l+n-1]}x):x\in Z(r(\eta)_{[k+n,k+n+m-1]})\}.
\end{equation*}
Then $U$ is an open neighbourhood of $\eta$ such that 
$$\phi(\xi)\in \{(r(\phi(\eta))_{[0,k'-1]}y,k'-l',s(\phi(\eta))_{[0,l'-1]}y):y\in Z(r(\phi(\eta))_{[k',k'+m'-1]})\}\subseteq V$$ 
if $\xi\in U$ and $s(\xi)$ is not an isolated point in $X$. We claim that $\phi(\xi)\in V$ if $\xi\in U$, also if $s(\xi)$ is an isolated point in $X$. Suppose $\xi\in U$ and that $s(\xi)$ is an isolated point in $X$. If $x$ is an isolated point in $X$ and periodic, then $Z(x_{[0,m-1]})=\{x\}$. It follows that $\osh^i(s(\xi))\ne x_{[s(\xi)]}$ for $i\in\{0,1,\dots,l+n-1\}$ and that $\osh^j(r(\xi))\ne x_{[s(\xi)]}$ for $j\in\{0,1,\dots,k+n-1\}$ because if $\osh^i(s(\xi))= x_{[s(\xi)]}$ and $i\in\{0,1,\dots,l+n-1\}$, then $\osh^i(s(\eta))\in Z((x_{[s(\xi)]})_{[0,m-1]})$, and if $\osh^j(r(\xi))= x_{[s(\xi)]}$ and $j\in\{0,1,\dots,k+n-1\}$, then $\osh^j(r(\eta))\in Z((x_{[s(\xi)]})_{[0,m-1]})$. Thus, $j_{r(\xi)}-j_{s(\xi)}=k-l$ and $n_\xi=0$. Similarly, $j_{h(r(\xi))}-j_{h(s(\xi))}=k'-l'$, so 
\begin{multline*}
	\phi(\xi)=(h(r(\xi)),k'-l',h(s(\xi)))\\\in \{(r(\phi(\eta))_{[0,k'-1]}y,k'-l',s(\phi(\eta))_{[0,l'-1]}y):y\in Z(r(\phi(\eta))_{[k',k'+m'-1]})\}\subseteq V.
\end{multline*} 
This shows that $\phi$ is continuous. That $\phi^{-1}$ is continuous can be proved in a similar way.

Thus, $\phi$ is an isomorphism such that $r(\phi(\eta))=h(r(\eta))$ and $s(\phi(\eta))=h(s(\eta))$ for $\eta\in\G_X$ and such that $\phi((x,\lp(x),x))=(h(x),\lp(h(x)),h(x))$ for every isolated point in $X$. Finally, it follows from Lemma~\ref{henrik} that $\phi((x,\lp(x),x))=(h(x),\lp(h(x)),h(x))$ for every eventually periodic point that is not an isolated point in $X$.
\end{proof}

\begin{remark}\label{vm}
Let $X$ and $Y$ be two one-sided shifts of finite type and let $\phi:\G_X\to\G_Y$ be an isomorphism. Then the map $h:X\to Y$ given by $\phi((x,0,x))=(h(x),0,h(x))$ is a continuous orbit equivalence (see Proposition~\ref{john}), and it follows from Lemma~\ref{henrik} that $\phi((x,\lp(x),x))=(h(x),\lp(h(x)),h(x))$ for every eventually periodic point $x\in X$ that is not isolated in $X$, but it might be the case that $\phi((x,\lp(x),x))=(h(x),-\lp(h(x)),h(x))$ if $x$ is isolated in $X$. 
\end{remark}

We have shown in Proposition~\ref{sven} that if two one-sided shifts of finite type $X$ and $Y$ are continuously orbit equivalent by an orbit equivalence $h:X\to Y$, then there is an isomorphism $\phi:\G_X\to \G_Y$ such that $r(\phi(\eta))=h(r(\eta))$ and $s(\phi(\eta))=h(s(\eta))$ for $\eta\in\G_X$ and such that $\phi((x,\lp(x),x))=(h(x),\lp(h(x)),h(x))$ for every eventually periodic point $x\in X$. From this, it is not difficult to construct the least period preserving cocycle pairs for $h$ and $h^{-1}$ we need to apply Proposition~\ref{diego} (see the proof of Proposition~\ref{john}). To construct the continuous maps $b:X\to\ZZ$, $n:X\to\NN_0$, $b':Y\to\ZZ$ and $n':Y\to\NN_0$ required to apply Proposition~\ref{diego}, we need the following proposition which is a generalisation of \cite[Proposition 3.4]{MM} to general shifts of finite type. The proof of Proposition~\ref{erik} is essentially identical to the proof of \cite[Proposition 3.4]{MM}, but we have included it for completeness.
 
\begin{proposition}\label{erik}
Let $X$ be a shift of finite type. Then there is an isomorphism $\Phi:H^1(\G_X)\to H^X$ such that $\Phi([f])=[g]$ where $g\in C(X,\ZZ)$ is given by $g(x)=f((x,1,\osh(x)))$. Moreover, $\Phi([f])\in H^X_+$ if and only if $f((x,\lp(x),x))\ge 0$ for every eventually periodic point $x\in X$. 
\end{proposition}

\begin{proof}
It is straightforward to check that $[f]\mapsto [g]$ where $g\in C(X,\ZZ)$ is given by $g(x)=f((x,1,\osh(x)))$, defines an isomorphism $\Phi:H^1(\G_X)\to H^X$, with inverse $\Phi^{-1}([g])((x,r-s,y))=\sum_{i=0}^{r-1}g(\osh^i(x))-\sum_{j=0}^{s-1}g(\osh^j(y))$ for every $(x,r-s,y)\in\G_X$. It is also easy to check that if $\Phi([f])\in H^X_+$, then $f((x,\lp(x),x))\ge 0$ for every eventually periodic point $x\in X$. 

Suppose $f\in\hom(\G,\ZZ)$ and that $f((x,\lp(x),x))\ge 0$ for every eventually periodic point $x\in X$. We shall prove that $\Phi([f])\in H^X_+$ by constructing $h\in C(X,\ZZ)$ such that $f((x,1,\osh(x)))+h(x)-h(\osh(x))\ge 0$ for all $x\in X$. Since $X$ is a shift of finite type and $f$ is continuous, there is an $m\ge 2$ such that if $v\in\LL(X)$ has length $m$ and $uv,vw\in\LL(X)$, then $uvw\in\LL(X)$, and such that if $x_{[0,m]}=x'_{[0,m]}$, then $f((x,1,\osh(x)))=f((x',1,\osh(x')))$. Let $E=(E^0,E^1,r,s)$ be the finite directed graph with $E^0=\{v\in\LL(X):v\text{ has length }m\}$, $E^1=\{w\in\LL(X):w\text{ has length }m+1\}$, $s(w)=w_{[0,m-1]}$ and $r(w)=w_{[1,m]}$ for $v\in E^1$. Let $\omega:E^1\to\ZZ$ be defined by $\omega(w)=f((x,1,\osh(x)))$ for some $x\in X$ with $x_{[0,m]}=w$. Let $w^1w^2\dots w^n$ be a cycle on $E$ (so $w^1,w^2,\dots,w^n\in E^1$, $r(w^i)=s(w^{i+1})$ for $i=1,2,\dots,n-1$, and $r(w^n)=s(w^1)$). Then there is a periodic $x\in X$ such that $x_{i+kn}=w^{i+1}_1$ ($w^{i+1}_1$ is the first letter of $w^{i+1}$) for $i=0,1,\dots,n-1$ and $k\in\NN_0$. It follows that there is $j\in\NN$ such that $n=j\lp(x)$ and that
\begin{equation*}
	\sum_{i=1}^n\omega(w^i)=\sum_{i=1}^nf((\osh^{i-1}(x),1,\osh^i(x)))=jf((x,\lp(x),x))\ge 0.
\end{equation*}
It therefore follows from \cite[Proposition 3.3(2)]{BH} that there is a function $\kappa:E^0\to\ZZ$ such that $\omega(w)+\kappa(s(w))-\kappa(r(w))\ge 0$ for $w\in E^1$. Define $h:X\to\ZZ$ by $h(x)=\kappa(x_{[0,m-1]})$. Then $h\in C(X,\ZZ)$ and
\begin{equation*}
f((x,1,\osh(x)))+h(x)-h(\osh(x))=\omega(x_{[0,m]})+\kappa(s(x_{[0,m]}))-\kappa(r(x_{[0,m]}))\ge 0
\end{equation*}
for $x\in X$.
\end{proof}

\begin{proposition}\label{john}
Let $X$ and $Y$ be two one-sided shifts of finite type and suppose that $\phi:\G_X\to\G_Y$ is an isomorphism. Then there is a continuous orbit equivalence $h:X\to Y$, a $h$-cocycle pair $(k,l)$ such that 
\begin{equation}\label{eq:5}
	\begin{split}
		\phi\bigl((&x,r-s,x')\bigr)\\&=\left(h(x),\sum_{i=0}^{r-1}[l(\osh^i(x))-k(\osh^i(x))]-\sum_{j=0}^{s-1}[l(\osh^j(x'))-k(\osh^j(x'))],h(x')\right)
	\end{split}
\end{equation}
for $(x,r-s,x')\in\G_X$ with $\osh^r(x)=\osh^s(x')$, and a $h^{-1}$-cocycle pair $(k',l')$ such that
\begin{equation}\label{eq:6}
	\begin{split}
		\phi^{-1}&\bigl((y,r'-s',y')\bigr)\\&=\left(h^{-1}(y),\sum_{i=0}^{r'-1}[l'(\osh[Y]^i(y))-k'(\osh[Y]^i(y))]-\sum_{j=0}^{s'-1}[l'(\osh[Y]^j(y'))-k'(\osh[Y]^j(y'))],h^{-1}(y')\right)
	\end{split}
\end{equation}
for $(y,r'-s',y')\in\G_Y$ with $\osh[Y]^{r'}(y)=\osh[Y]^{s'}(y')$.

The $h$-cocycle pair $(k,l)$ and the $h^{-1}$-cocycle pair $(k',l')$ are least period preserving if and only if $\phi((x,\lp(x),x))=(h(x),\lp(h(x)),h(x))$ for every eventually periodic $x\in X$, in which case there exist continuous maps $b,n:X\to\NN_0$ and $b',n':Y\to\NN_0$ such that $l(x)-k(x)=n(x)+b(x)-b(\osh(x))$ and $l'(y)-k'(y)=n'(y)+b'(y)-b'(\osh[Y](y))$ for $x\in X$ and $y\in Y$.
\end{proposition}

\begin{proof}
The restriction of $\phi$ to $\G^{(0)}=X$ is a homeomorphism $h:X\to Y$ such that $r(\phi(\eta))=h(r(\eta))$ and $s(\phi(\eta))=h(s(\eta))$ for $\eta\in\G_X$. We shall prove that $h$ is a continuous orbit equivalence by constructing an $h$-cocycle pair $(k,l)$ satisfying \eqref{eq:5} and an $h^{-1}$-cocycle pair $(k',l')$ satisfying \eqref{eq:6}.

Let $x\in X$ be a point that is not isolated in $X$, and let $A$ be an open bisection containing $(\sigma(x),-1,x)$. Then $\phi(A)$ is an open bisection containing $\phi((\sigma(x),-1,x))$, so the map $\alpha_{\phi(A)}:s(\phi(A))\to r(\phi(A))$ defined by $\alpha_{\phi(A)}(s(\xi))=r(\xi)$ for $\xi\in \phi(A)$ is a homeomorphism with the property that there are continuous maps $k,l:s(\phi(A))\to\NN_0$ such that $\osh[Y]^{k(y)}(\alpha_{\phi(A)}(y))=\osh[Y]^{l(y)}(y)$ for every $y\in s(\phi(A))$ (cf. \cite[Proposition 3.3]{BCW}). It follows from Lemma~\ref{kurt} and the fact that $Y$ is a locally compact and totally disconnected Hausdorff space that there is an clopen neighbourhood $V$ of $h(x)$ and $k_x,l_x\in\NN_0$ such that $\osh[Y]^{k_x}(\alpha_{\phi(A)}(y))=\osh[Y]^{l_x}(y)$ for $y\in V$. Let $U_x=h^{-1}(V)$. Then $U_x$ is a clopen neighbourhood of $x$ and 
\begin{equation*}
\osh[Y]^{k_x}(h(\osh(x')))
=\osh[Y]^{k_x}(\alpha_{\phi(A)}(h(x')))
=\osh[Y]^{l_x}(h(x'))
\end{equation*}
for all $x'\in U_x$. If $x\in X$ is isolated in $X$, then we let $U_x=\{x\}$ and choose $k_x,l_x\in\NN_0$ such that $\osh[Y]^{k_x}(h(\osh(x)))=\osh[Y]^{l_x}(h(x))$ (we can do that because $\phi((\osh(x),-1,x))\in\G_Y$). 

Since $X$ is compact, it follows that there is a finite $F\subseteq X$ and mutually disjoint clopen sets $\{U'_x:x\in F\}$ such that $x\in U'_x\subseteq U_x$ for $x\in F$ and $\bigcup_{x\in F}=X$. If we define $k,l:X\to\NN_0$ by setting $k(x)=k_{x'}$ and $l(x)=l_{x'}$ for $x\in U'_{x'}$, then $(k,l)$ is an $h$-cocycle pair satisfying \eqref{eq:5}. We can in a similar way construct an $h^{-1}$-cocycle pair $(k',l')$ satisfying \eqref{eq:6}.

It follows from \eqref{eq:5} that if $x\in X$ is eventually periodic, then 
\begin{equation*}
\phi((x,\lp(x),x))=\left(h(x),\sum_{i=0}^{\lp(x)-1}(l(\osh^i(x))-k(\osh^i(x))),h(x)\right),
\end{equation*}
and it follows from \eqref{eq:6} that if $y\in Y$ is eventually periodic, then 
\begin{equation*}
\phi^{-1}((y,\lp(y),y))=\left(h^{-1}(y),\sum_{i=0}^{\lp(y)-1}(l'(\osh[Y]^i(y))-k'(\osh[Y]^i(y))),h^{-1}(y)\right).
\end{equation*}
It follows that $(k,l)$ and $(k',l')$ are least period preserving if and only if $\phi((x,\lp(x),x))=(h(x),\lp(h(x)),h(x))$ for every eventually periodic $x\in X$.

It follows from Proposition~\ref{erik} and \eqref{eq:5} that there is an isomorphism $\Psi:H^Y\to H^X$ such that $\Psi([f])=[g]$ where $g\in C(X,\ZZ)$ is given by 
\begin{equation*}
g(x)=\sum_{i=0}^{l(x)-1}f(\osh[Y]^i(h(x)))-\sum_{j=0}^{k(x)-1}f(\osh[Y]^j(h(\osh(x)))).
\end{equation*}
Suppose $\phi((x,\lp(x),x))=(h(x),\lp(h(x)),h(x))$ for every eventually periodic $x\in X$. It then follows from Proposition~\ref{erik} that $\Psi(H^Y_+)=H^X_+$. Let $1_Y$ be the function that sends every $y\in Y$ to $1$. Then $[1_Y]\in H^Y_+$, so $\Psi([1_Y])\in H^X_+$. Since $\Psi([1_Y])=[l-k]$, it follows that there are continuous maps $b,n:X\to\NN_0$ such that $l(x)-k(x)=n(x)+b(x)-b(\osh(x))$ for $x\in X$. The existence of continuous maps $b',n':Y\to\NN_0$ such that $l'(y)-k'(y)=n'(y)+b'(y)-b'(\osh[Y](y))$ for $y\in Y$, can be proved in a similar way.
\end{proof}

\begin{proof}[Proof of Theorem \ref{ib}]
Let $h:X\to Y$ be a continuous orbit equivalence. It follows from Proposition~\ref{sven} that $h$ maps eventually periodic points to eventually periodic points, and that there is an isomorphism $\phi:\G_X\to \G_Y$ such that $r(\phi(\eta))=h(r(\eta))$ and $s(\phi(\eta))=h(s(\eta))$ for $\eta\in\G_X$ and such that $\phi((x,\lp(x),x))=(h(x),\lp(h(x)),h(x))$ for every eventually periodic point $x\in X$. It follows from Proposition~\ref{john} that there are a least period preserving $h$-cocycle pair $(k,l)$, a least period preserving $h^{-1}$-cocycle pair $(k',l')$, and continuous maps $b,n:X\to\NN_0$ and $b',n':Y\to\NN_0$ such that $l(x)-k(x)=n(x)+b(x)-b(\osh(x))$ and $l'(y)-k'(y)=n'(y)+b'(y)-b'(\osh[Y](y))$ for $x\in X$ and $y\in Y$. It therefore follows from Proposition~\ref{diego} that $\X$ and $\Y$ are flow equivalent.
\end{proof}

\section{Flow equivalence and orbit equivalence for shifts of finite type and isomorphisms of their groupoids}

The following theorem follows directly from Proposition~\ref{sven} and Proposition~\ref{john} (it also follows from \cite[Corollary 4.6]{CW} and the fact that a shift of finite type can be represented by a finite graph that has no sinks). 

\begin{theorem}\label{orbit}
Let $X$ and $Y$ be two one-sided shifts of finite type. Then $X$ and $Y$ are continuously orbit equivalent if and only if the groupoids $\G_X$ and $\G_Y$ are isomorphic.
\end{theorem}

If $X$ and $Y$ are irreducible, then the result of Theorem~\ref{orbit} easily follows from \cite[Theorem 2.3]{MM} and the fact that every one-sided shift of finite type is conjugate to a one-sided topological Markov shift. If $X$ and $Y$ have no isolated periodic points, then the result of Theorem~\ref{orbit} follows from \cite[Proposition 3.2]{Renault}.

In the rest of this section we prove an analogue of Theorem~\ref{orbit} (Theorem~\ref{thm:1}), which, among other things, says that the groupoids of two one-sided shifts of finite type are stably isomorphic if and only if the corresponding two-sided shift spaces are flow equivalence. We will use results of \cite{CRS}, \cite{ERRS}, and \cite{Matui} for this.

Before we state and prove Theorem~\ref{thm:1}, we need to introduce some notation and a lemma.

Suppose $X$ is a one-sided shift space and $f\in C(X,\NN)$. Let
$$X_f:=\{(x,i)\in X\times\NN_0:i<f(x)\, \},$$ 
and equip $X_f$ with the subspace topology of $X\times\NN_0$, where the latter is equipped with the product topology of the topology of $X$ and the discrete topology on $\NN_0$. Then $X_f$ is compact. Define $\osh[X_f]:X_f\to X_f$ by 
$$\osh[X_f](x,0)=(\osh(x),f(\osh(x))-1\, )$$ 
and $\osh[X_f](x,i)=(x,i-1)$ for $x\in X$ and $i\in\{1,2,\dots,f(x)-1\}$. Then $\osh[X_f]$ is continuous and surjective.

Let $\alp$ be the alphabet of $X$ and $\mathcal N:=\{0,1,\dots,\max\{f(x):x\in X\}-1\}$. Define $\iota_{X_f}:X_f\to (\alp\times\mathcal{N})^{\NN_0}$ by 
\begin{multline*}
	\iota_{X_f}(x,i)=(a_0,i)(a_0,i-1)\dots (a_0,0)(a_1,f(\osh(x))-1)(a_1,f(\osh(x))-2)\\\dots (a_1,0)(a_2,f(\osh^2(x))-1)\dots
\end{multline*}
where $x=a_0a_1a_2\dots$. Then $\iota_{X_f}$ is continuous and injective and $\iota_{X_f}\circ\osh[X_f]=\sigma\circ\iota_{X_f}$, where $\sigma$ is the shift map on $(\alp\times\mathcal{N})^{\NN_0}$. It follows that $X^f:=\iota_{X_f}(X_f)$ is a subshift of $(\alp\times\mathcal{N})^{\NN_0}$, and that $(X_f,\osh[X_f])$ and $(X^f,\osh[X^f])$ are conjugate.

Let $X_0:=\iota_{X_f}(X\times\{0\})$ and define $\iota_X:X\to X_0$ by $\iota_X(x)=\iota_{X_f}(x,0)$. Then $X_0$ is a cross section of $X^f$ (i.e., $X_0$ is a compact and open subset of $X^f$, $X^f=\{\osh[X^f]^k(x):x\in X_0,\ k\in\NN_0\}$, there exists for all $x\in X_0$ a $k\in\NN$ such that $\osh[X^f]^k(x)\in X_0$, and $\frt_{X_0}:X_0\to\NN$ defined by $\frt_{X_0}(x)=\min\{k\in\NN:\osh[X^f]^k(x)\in X_0\}$ is continuous), and $\iota_X$ is a conjugacy between $(X,\osh)$ and $(X_0,\osh[X_0])$ where $\osh[X_0]$ is the \emph{first return map} defined by $\osh[X_0](x)=\osh[X^f]^{\frt_{X_0}(x)}(x)$. It follows that the two-sided shift spaces $\X$ and $\X^f$ are flow equivalent. So $X^f$ is of finite type if and only if $X$ is. 

If $X$ is of finite type, then we let $(\G_X)_f$ denote the groupoid 
 \begin{equation*}
 	\{(\eta,i,j)\in\G_X\times\NN_0\times\NN_0: 0\le i<f(r(\eta)),\ 0\le j<f(s(\eta))\, \},
 \end{equation*}
introduced in \cite{Matui}.

\begin{lemma}\label{lem:groupoid}
	Let $X$ be a one-sided shift space and $f\in C(X,\NN)$. Then $(\G_X)_f$ and $\G_{X^f}$ are isomorphic.
\end{lemma}

\begin{proof}
	It is routine to check that the map 
$$((x,l-k,y),i,j)\mapsto \left(\iota_{X_f}(x,i),i-j+\sum_{h=1}^{l}f(\osh^h(x))-\sum_{h=1}^{k}f(\osh^h(y)),\iota_{X_f}(y,j)\right)$$ 
where $x,y\in X$, $i,j,k,l\in\NN_0$, $i<f(x)$, $j<f(y)$, and $\osh^l(x)=\osh^k(y)$, is an isomorphism between $(\G_X)_f$ and $\G_{X^f}$.
\end{proof}

Suppose $\G$ is a groupoid with unit space $\G^{(0)}$ and range and source maps $r,s:\G\to\G^{(0)}$, and that $Z$ is a subset of $\G^{(0)}$. We let $\G|_Z:=s^{-1}(Z)\cap r^{-1}(Z)$, and say that $Z$ is \emph{$\G$-full} or just \emph{full} if $r(s^{-1}(Z))=\G^{(0)}$. Suppose $\G_1$ and $\G_2$ are two ample groupoids (see for example \cite{CRS}). As in \cite{CRS} and \cite{Matui}, we say that $\G_1$ and $\G_2$ are \emph{Kakutani equivalent} if there for $i=1,2$ is a $\G_i$-full clopen subset $Z_i\subseteq\G_i^{(0)}$ such that $\G_1|_{Z_1}$ and $\G_2|_{Z_2}$ are isomorphic as topological groupoids, and we say that $\G_1$ and $\G_2$ are \emph{groupoid equivalent} if there is a $\G_1$--$\G_2$ equivalence in the sense of \cite[Definition 2.1]{MRW}.  

As in \cite{CRS}, we write $\mathcal{R}$ for the full countable equivalence relation $\NN_0\times\NN_0$ regarded as a discrete principal groupoid with unit space $\NN_0$.

\begin{theorem}\label{thm:1}
	Let $X$ and $Y$ be one-sided shifts of finite type. The following are equivalent.
	\begin{enumerate}
		\item $\G_X$ and $\G_Y$ are Kakutani equivalent.
		\item $\G_X$ and $\G_Y$ are groupoid equivalent.
		\item $\G_X\times\mathcal{R}$ and $\G_Y\times\mathcal{R}$ are isomorphic as topological groupoids.
		\item There exist full open sets $X'\subseteq X$ and $Y'\subseteq Y$ such that $\G_X|_{X'}$ and $\G_Y|_{Y'}$ are isomorphic as topological groupoids.
		\item There exist $f\in C(X,\NN)$ and $g\in C(Y,\NN)$ such that $X^f$ and $Y^g$ are continuous orbit equivalent.
		\item $\X$ and $\Y$ are flow equivalent. 
	\end{enumerate}
\end{theorem}

\begin{proof}
The equivalence of (1)--(4) follows from \cite[Theorem 3.2]{CRS}.
	
(1)$\implies$(5): It follows from \cite[Lemma 4.6]{Matui} that there are $f\in C(X,\NN)$ and $g\in C(Y,\NN)$ such that $(\G_X)_f$ and $(\G_Y)_g$ are isomorphic. Since $(\G_X)_f$ is isomorphic to $\G_{X^f}$, and $(\G_Y)_g$ is isomorphic to $\G_{Y^g}$, it follows that $\G_{X^f}$ and $\G_{Y^g}$ are isomorphic, so an application of Theorem~\ref{orbit} gives us that $X^f$ and $Y^g$ are continuous orbit equivalent.
	
(5)$\implies$(6): It follows from Theorem~\ref{ib} that $\X^f$ and $\Y^g$ are flow equivalent. Since $\X$ and $\X^f$ are flow equivalent, and $\Y$ and $\Y^g$ are flow equivalent, it follows that $\X$ and $\Y$ are flow equivalent.
	
(6)$\implies$(3): This is probably well-known to the experts, but we have not been able to find a proof in the literature, so we give one here for completeness. Let $E_1$ and $E_2$ be finite directed graphs with no sinks and no sources such that the one-sided edge shift of $E_1$ is conjugate to $X$ and the one-sided edge shift of $E_2$ is conjugate to $Y$. Then the two-sided edge shifts of $E_1$ and $E_2$ are flow equivalent, so it follows from \cite[Lemma 5.1]{ERRS} that $E_1$ and $E_2$ are move equivalent. Since the groupoid of $E_1$ is isomorphic to $\G_X$ and the groupoid of $E_2$ is isomorphic to $\G_Y$, an application of \cite[Corollary 4.9]{CRS} gives that $\G_X\times\mathcal{R}$ and $\G_Y\times\mathcal{R}$ are isomorphic as topological groupoids. 
\end{proof}

\section{Graph $C^*$-algebras and Leavitt path algebras}

We now apply the results of the previous section to strengthen some of the results of \cite{ABHS,BCaH,BCW,CRS} in the special case of finite directed graphs with no sinks and no sources.

Suppose $E$ is a directed graph and $R$ is a unital ring. We write $\G_E$ for the groupoid of $E$ (see for example \cite{AER,BCW,Carlsen,CW,CS}), $C^*(E)$ for the $C^*$-algebra of $E$ (see for example \cite{AER,BCW,Carlsen,ERRS,Sor,Tomforde}), $D(E)$ for the $C^*$-subalgebra $\overline{\operatorname{span}}\{s_\mu s_\mu^*:\mu\in E^*\}$ of $C^*(E)$, $L_R(E)$ for the Leavitt path $R$-algebra of $E$ (see for example \cite{Carlsen,Tomforde2}), $D_R(E)$ for the $*$-subalgebra $\operatorname{span}_R\{\mu \mu^*:\mu\in E^*\}$ of $L_R(E)$, and $X_E$ for the one-sided edge shift of $E$.

The next result depends on the equivalence of orbit equivalence as defined in \cite{BCW} and isomorphisms of the associated groupoids. This was established in \cite{BCW} in the presence of the so-called Condition (L), and generalised to the given setting in the independent articles \cite{AER} and \cite{CW}. We will follow \cite{CW} since it is much better aligned with the approach in the paper at hand. Combining Theorem 5.1 further with results of \cite{BCW} and \cite{CR2}, we obtain the following corollary:

\begin{corollary}\label{cor:2}
Let $E$ and $F$ be finite directed graphs with no sinks and no sources and let $R$ be a unital commutative integral domain. The following are equivalent.
\begin{enumerate}
	\item $X_E$ and $X_F$ are continuous orbit equivalent.
	\item $\G_E$ and $\G_F$ are isomorphic.
	\item There is a $*$-isomorphism $\phi:C^*(E)\to C^*(F)$ such that $\phi(D(E))=D(F)$.
	\item There is a ring isomorphism $\beta:L_R(E)\to L_R(F)$ such that $\beta(D_R(E))=D_R(F)$.		
	\item There is a $*$-algebra isomorphism $\gamma:L_R(E)\to L_R(F)$ such that $\gamma(D_R(E))=D_R(F)$.
	\item $E$ and $F$ are orbit equivalent as defined in \cite{BCW}.
\end{enumerate}
\end{corollary}

\begin{proof}
The equivalence of (1) and (2) is proved in Theorem~\ref{orbit}. The equivalence of (2) and (3) follows from \cite[Theorem 5.1]{BCW}, and the equivalence of (2), (4), and (5) follows from \cite[Corollary 4.2]{CR2}. Finally, the equivalence of (2) and (6) follows from \cite[Corollary 4.6]{CW}.
\end{proof}

\begin{remark}
It follows from \cite[Corollary 6]{Carlsen} that if $R=\mathbb{Z}$ (or more generally, if $R$ is a \emph{kind} $*$-subring of $\CC$, see \cite[Section 3]{Carlsen}), then the condition ``$L_R(E)$ and $L_R(F)$ are isomorphic $*$-rings'' can be added to the list of equivalent conditions in Corollary \ref{cor:2}.
\end{remark}

Suppose $E$ is a directed graph. We write $\X_E$ for the two-sided edge shift of $E$, and as in \cite{CRS} and \cite{CW}, we write $SE$ for the directed graph obtained by attaching a head to each vertex of $E$ (see \cite[Definition 4.2]{Tomforde}). We write $\mathcal{K}$ for the $C^*$-algebra of compact operators on $l^2(\NN_0)$ and $\mathcal{C}$ for the maximal abelian subalgebra of $\mathcal{K}$ consisting of diagonal operators. 

Suppose $R$ is a unital ring. We write $M_\infty(R)$ for the ring of finitely supported, countable infinite square matrices over $R$, and $D_\infty(R)$ for the abelian subring of $M_\infty(R)$ consisting of diagonal matrices.

By combining Theorem~\ref{thm:1} with results of \cite{CR2,CRS,CW,ERRS}, we obtain the following corollary.

\begin{corollary}\label{cor:1}
Let $E$ and $F$ be finite directed graphs with no sinks and no sources and let $R$ be a unital commutative integral domain. The following are equivalent.
	
	\begin{enumerate}
		\item $E$ and $F$ are move equivalent as defined in \cite{Sor}.
		\item $\X_E$ and $\X_F$ are flow equivalent.
		\item $\G_E$ and $\G_F$ are Kakutani equivalent.
		\item $\G_E$ and $\G_F$ are groupoid equivalent.
		\item $\G_E\times\mathcal{R}$ and $\G_E\times\mathcal{R}$ are isomorphic as topological groupoids.
		\item $\G_{SE}$ and $\G_{SF}$ are isomorphic as topological groupoids.
		\item $SE$ and $SF$ are orbit equivalent as defined in \cite{BCW}.
		\item There is a $*$-isomorphism $\phi:C^*(E)\otimes\mathcal{K}\to C^*(F)\otimes\mathcal{K}$ such that $\phi(D(E)\otimes\mathcal{C})=D(F)\otimes\mathcal{C}$.
		\item There is a ring isomorphism $\beta:L_R(E)\otimes M_\infty(R)\to L_R(F)\otimes M_\infty(R)$ such that $\beta(D_R(E)\otimes D_\infty(R))=D_R(F)\otimes D_\infty(R)$.
		\item There is a $*$-algebra isomorphism $\gamma:L_R(E)\otimes M_\infty(R)\to L_R(F)\otimes M_\infty(R)$ such that $\gamma(D_R(E)\otimes D_\infty(R))=D_R(F)\otimes D_\infty(R)$.
		\item There is a $*$-isomorphism $\psi:C^*(SE)\to C^*(SF)$ such that $\psi(D(SE))=D(SF)$.
		\item There is a ring isomorphism $\rho:L_R(SE)\to L_R(SF)$ such that $\rho(D_R(SE))=D_R(SF)$.
		\item There is a $*$-algebra isomorphism $\tau:L_R(SE)\to L_R(SF)$ such that $\tau(D_R(SE))=D_R(SF)$.
		\item There exist projections $p_E\in D(E)$ and $p_F\in D(F)$ and a $*$-isomorphism $\omega:p_EC^*(E)p_E\to p_FC^*(F)p_F$ such that $p_E$ is full in $C^*(E)$, $p_F$ is full in $C^*(F)$, and $\omega(p_ED(E))=p_FD(F)$.
		\item There exist idempotents $p_E\in D_R(E)$ and $p_F\in D_R(F)$ and a ring isomorphism $\zeta:p_EL_R(E)p_E\to p_FL_R(F)p_F$ such that $p_E$ is full in $L_R(E)$, $p_F$ is full in $L_R(F)$, and $\zeta(p_ED_R(E))=p_FD_R(F)$.
		\item There exist projections $p_E\in D_R(E)$ and $p_F\in D_R(F)$ and a $*$-algebra isomorphism $\kappa:p_EL_R(E)p_E\to p_FL_R(F)p_F$ such that $p_E$ is full in $L_R(E)$, $p_F$ is full in $L_R(F)$, and $\kappa(p_ED_R(E))=p_FD_R(F)$.
	\end{enumerate}
\end{corollary}

The implication $(2)\implies (8)$ of Corollary~\ref{cor:1} was originally proved for irreducible graphs satisfying condition (L) by Cuntz and Krieger in \cite[Theorem 4.1]{CK} (in the setting of Cuntz--Krieger graphs), and for reducible graphs satisfying condition (L) by Cuntz in \cite[Theorem 2.4]{Cuntz} (also in the setting of Cuntz--Krieger graphs). If the graphs $E$ and $F$ both satisfy condition (L), then the equivalence of (3) and (14) of Corollary~\ref{cor:1} follows from \cite[Theorem 5.4]{Matui}.

\begin{proof}[Proof of Corollary~\ref{cor:1}]
The equivalence of (1) and (2) is proved in \cite[Lemma 5.1]{ERRS}. Since $\G_E=\G_{X_E}$ and $\G_F=\G_{X_F}$, the equivalence of (2)--(5) follows from Theorem~\ref{thm:1}. The equivalence of (5), (6), (8), and (11) follows from \cite[Theorem 4.2]{CRS}, the equivalence of (3) and (14) follows from \cite[Corollary 4.5]{CRS}, and the equivalence of (6) and (7) follows from \cite[Corollary 4.6]{CW}. Finally, the equivalence of (9)--(16) is proved in \cite[Corollary 4.8]{CR2}.
\end{proof}

\begin{remark}
It follows from \cite[Proposition 5]{Carlsen} and an argument similar to the one used in the proof of \cite[Proposition 6.1]{JS} that if $R=\mathbb{Z}$ (or more generally, if $R$ is a \emph{kind} $*$-subring of $\CC$, see \cite[Section 3]{Carlsen}), then the following 3 conditions can be added to the list of equivalent conditions in Corollary \ref{cor:1}.
	\begin{enumerate}
		\item[(17)] $L_R(E)\otimes M_\infty(R)$ and $L_R(F)\otimes M_\infty(R)$ are isomorphic $*$-rings.
		\item[(18)] $L_R(SE)$ and $L_R(SF)$ are isomorphic $*$-rings.
		\item[(19)] There are projections $p_E\in D_R(E)$ and $p_F\in D_R(F)$  such that $p_E$ is full in $L_R(E)$, $p_F$ is full in $L_R(F)$, and $p_EL_R(SE)p_E$ and $p_FL_R(SF)p_F$ are isomorphic $*$-rings.
	\end{enumerate}
\end{remark}

\section{Cuntz--Krieger algebras}

In this final section, we apply the results of Section 5 and Section 6 to Cuntz--Krieger algebras and topological Markov chains and directed graphs of $\{0,1\}$-matrix in order to generalise \cite[Theorem 2.3]{MM} and \cite[Corollary 3.8]{MM} from the irreducible to the general case.

Let $A$ be an $N\times N$-matrix with entries in $\{0,1\}$, and assume that every row and every column of $A$ is nonzero. As in \cite{CK,MM}, we denote by $\Oo_A$ the Cuntz--Krieger algebra of $A$ with canonical abelian subalgebra $\D_A$, by $X_A$ the one-sided shift space $\{(x_i)_{i\in\NN_0}\in\{1,2,\dots,N\}^{\NN_0}:A(x_i,x_{i+1})=1\text{ for all }i\in\NN_0\}$ associated with $A$, and by $\TSS$ the two-sided subshift of $X_A$ (if $A$ does not satisfy Condition (I), then we let $\Oo_A$ denote the universal Cuntz--Krieger algebra $\mathcal{A}\Oo_A$ introduced in \cite{aHR}). Notice that the groupoid $\G_{X_A}$ is equal to the groupoid $G_A$ considered in \cite{MM}. 

We get from Theorem~\ref{orbit} and \cite[Theorem 5.1]{BCW} (or Corollary~\ref{cor:2}) the following generalisation of \cite[Theorem 2.3]{MM}.

\begin{corollary}\label{CK}
Let $A$ and $B$ be finite square matrices with entries in $\{0,1\}$, and assume that every row and every column of $A$ and $B$ is nonzero. The following conditions are equivalent.
\begin{enumerate}
	\item $X_A$ and $X_B$ are continuously orbit equivalent.
	\item The groupoids $\G_{X_A}$ and $\G_{X_B}$ are isomorphic.
	\item There is a $*$-isomorphism $\Psi:\Oo_A\to\Oo_B$ such that $\Psi(\D_A)=\D_B$.
\end{enumerate}
\end{corollary}

\begin{proof}
The equivalence of (1) and (2) follows from Theorem~\ref{orbit}.
	
Let $E_A$ be the graph of $A$, i.e., $E_A^0$ is the index set of $A$, $E_A^1=\{(i,j)\in E_A^0\times E_A^0: A(i,j)=1\}$, and $r((i,j))=j$ and $s((i,j))=i$ for $(i,j)\in E_A^1$. Then $\G_{X_A}$ is isomorphic to the groupoid $\G_{E_A}$ of $E_A$ defined in \cite{BCW}. It is well-known that there is an isomorphism $\Psi:\Oo_A\to C^*(E_A)$ satisfying $\Psi(\D_A)=\D(E_A)$. The equivalence of (2) and (3) therefore follows from \cite[Theorem 5.1]{BCW} (or Corollary~\ref{cor:2}).
\end{proof}

Similarly, we get from Corollary~\ref{cor:1} the following strengthening of \cite[Corollary 3.8]{MM}.

\begin{corollary} \label{cor:3}
Let $A$ and $B$ be finite square matrices with entries in $\{0,1\}$, and assume that every row and every column of $A$ and $B$ is nonzero. The following are equivalent.
	\begin{enumerate}
		\item There is a $*$-isomorphism $\Psi:\Oo_A \otimes\mathcal{K}\to \Oo_B \otimes\mathcal{K}$ such that $\Psi(\mathcal{D}_{A}\otimes\mathcal{C})=\mathcal{D}_{B}\otimes\mathcal{C}$.
		\item There are projections $p_A\in \mathcal{D}_A$ and $p_B\in \mathcal{D}_B$ and an isomorphism $\phi:p_A\Oo_Ap_A\to p_B\Oo_Bp_B$ such that $p_A$ is full in $\Oo_A$, $p_B$ is full in $\Oo_B$, and $\phi(p_A\mathcal{D}_A)=p_B\mathcal{D}_B$.
		\item $\TSS$ and $\TSS[B]$ are flow equivalent.
		\item $\G_{X_A}$ and $\G_{X_B}$ are Kakutani equivalent.
		\item $\G_{X_A}$ and $\G_{X_B}$ are groupoid equivalent.
		\item $\G_{X_A}\times\mathcal{R}$ and $\G_{X_B}\times\mathcal{R}$ are isomorphic.
	\end{enumerate}
\end{corollary}

\begin{proof}
Let $E_A$ be as in the proof of Corollary~\ref{CK}. Then the two-sided shift spaces $\TSS$ and $\X_{E_A}$ are conjugate, the groupoids $\G_{E_A}$ and $\G_{X_A}$ are isomorphic, and there is a $*$-isomorphism $\Psi:\Oo_A\to C^*(E_A)$ satisfying $\Psi(\D_A)=D(E_A)$. The result therefore follows from Corollary~\ref{cor:1}.
\end{proof}

\end{document}